\documentclass[a4paper]{amsart}

\usepackage[usenames, dvipsnames]{xcolor}
\usepackage{enumitem}
\usepackage{amsmath, amssymb, amsthm}
\usepackage[all]{xy}
\usepackage{tikz}
\usepackage{tikz-cd}
\usepackage{braket}
\usepackage{graphicx}
\usepackage[mathscr]{eucal}

\usepackage[T1]{fontenc}
\usepackage{lmodern}
\usepackage{hyperref}
\usepackage{microtype}

\newtheorem{theorem}{Theorem}[subsection]
\newtheorem{lemma}[theorem]{Lemma}

\newtheorem*{gaftA}{Theorem}
\newtheorem*{proposition*}{Proposition}
\newtheorem*{corollary*}{Corollary}
\newtheorem*{gaftB}{Theorem}
\newtheorem{corollary}[theorem]{Corollary}
\newtheorem{proposition}[theorem]{Proposition}

\theoremstyle{definition}
\newtheorem{definition}[theorem]{Definition}
\newtheorem*{definition*}{Definition}

\newtheorem{remark}[theorem]{Remark}
\newtheorem{example}[theorem]{Example}

\newcommand{\C}{\mathscr{C}}
\newcommand{\CC}{\mathfrak{C}}
\newcommand{\A}{\mathscr{A}}
\newcommand{\B}{\mathscr{B}}
\newcommand{\h}{\mathrm{h}}
\newcommand{\MM}{\mathscr{M}}
\newcommand{\NN}{\mathscr{N}}
\newcommand{\D}{\mathscr{D}}
\renewcommand{\Set}{\mathrm{Set}}
\newcommand{\N}{\mathbb N}
\newcommand{\op}{\mathrm{op}}
\newcommand{\colim}{\mathrm{colim}}
\newcommand{\map}{\mathrm{map}}
\newcommand{\Map}{\mathrm{Map}}

\newcommand{\Sn}{\mathcal{S}_{< n}}
\renewcommand{\S}{\mathcal{S}}
\newcommand{\id}{\mathrm{id}}
\newcommand{\sk}{\mathrm{sk}}
\newcommand{\Idem}{\mathrm{Idem}}
\newcommand{\Nec}{\mathcal{N}ec}
\newcommand{\SSet}{\mathscr{SS}et}
\newcommand{\SCat}{\mathscr{SC}at}

\title[Adjoint Functor Theorems and Higher Brown Representability]{Higher weak (co)limits, Adjoint functor theorems, and higher Brown representability}	
\author{Hoang Kim Nguyen}
\address{\newline
H.~K.~Nguyen \newline
Fakult\"{a}t f\"ur Mathematik \\
Universit\"{a}t Regensburg \\
93040 Regensburg, Germany}
\email{hoang-kim.nguyen@ur.de} 
\author{George Raptis}
\address{\newline
G. Raptis \newline
Fakult\"{a}t f\"ur Mathematik \\
Universit\"{a}t Regensburg \\
93040 Regensburg, Germany}
\email{georgios.raptis@ur.de}
\author{Christoph Schrade}
\address{\newline
C. Schrade \newline
Mathematisches Institut \\
WWU M\"{u}nster \\
48149 M\"{u}nster, Germany}
\email{schrade.christoph@gmail.com}
\begin{document}
\begin{abstract}
We prove general adjoint functor theorems for weakly (co)complete $n$-categories. This class of $n$-categories includes the homotopy $n$-categories of (co)complete $\infty$-categories, so these $n$-categories do not admit all small (co)limits in general. We also introduce Brown representability for (homotopy) $n$-categories and prove a Brown representability theorem for localizations of compactly generated $n$-categories. This class of $n$-categories includes the homotopy $n$-categories of presentable $\infty$-categories if $n \geq 2$, and the homotopy $n$-categories of presentable stable $\infty$-categories for any $n \geq 1$.  
\end{abstract}

\maketitle
\setcounter{tocdepth}{1} 
\tableofcontents

\section{Introduction}

Adjoint functor theorems (AFTs) typically characterize left (resp. right) adjoint functors which are defined on suitable cocomplete (resp. complete) categories. A natural necessary (and sometimes also sufficient) condition in these theorems is that the functor preserves small colimits (resp. small limits). 
Familiar and important examples of this type of AFT in ordinary category theory include: Freyd's \emph{General} and \emph{Special Adjoint Functor Theorems} (GAFT and SAFT) (see, for example, \cite{ML}) and the \emph{Left} and \emph{Right Adjoint Functor Theorems} in the context of locally presentable categories (see \cite{AR}). 

Most of these adjoint functor theorems have been generalized to the context of $\infty$-categories \cite{HTT, NRS}. AFTs are especially useful in higher category theory because providing an adjoint by an explicit construction is practically impossible in general in this context -- when such a construction is possible, this usually arises from an explicit model-dependent construction in a stricter context (e.g. model categories, simplicial categories, etc.). 
Lurie \cite{HTT} generalized the AFTs for locally presentable categories to presentable $\infty$-categories. In our previous work \cite{NRS}, we proved two versions of the GAFT for $\infty$-categories \cite[Theorems 3.2.5 and 3.2.6]{NRS}, a special version of the SAFT \cite[Theorem 4.1.3]{NRS}, and we showed that these also recover Lurie's AFTs for presentable $\infty$-categories \cite[Section 4]{NRS}. All these results are still in the context of (finitely) cocomplete (resp. complete) $\infty$-categories. 

On the other hand, a different type of adjoint functor theorem arises in the context of Brown representability \cite{Br, He, NRS}. These AFTs characterize left adjoint functors which are defined on suitable ordinary categories that are typically not finitely cocomplete. Important examples of such categories, which \emph{satisfy Brown representability}, include the homotopy categories of presentable stable $\infty$-categories \cite{HA, NRS}. These categories admit small coproducts, but only weak pushouts in general. The AFT in this case characterizes left adjoint functors in terms of the preservation of small coproducts and weak pushouts (see, for example, \cite[Section 5]{NRS}). 

\medskip

The question naturally arises whether there are more general AFTs which apply to higher (= $(n,1)$-)categories that are not necessarily finitely cocomplete (resp. complete). Our first goal in this paper is to prove generalizations of the GAFTs in \cite{NRS} to suitable $n$-categories (= $(n, 1)$-categories) that are not necessarily finitely cocomplete (resp. finitely complete).  The motivating example is the homotopy $n$-category of a (finitely) cocomplete (resp. complete) $\infty$-category. We recall that the homotopy $n$-category $\h_n\C$ of an $\infty$-category $\C$ is the $n$-category which is obtained from $\C$ after truncating its mapping spaces at level $n$ -- this is the usual homotopy category when $n=1$. The general construction and the properties of the homotopy $n$-category were studied in \cite{HTT}. Homotopy $n$-categories were studied further in \cite{Ra} in connection with a higher categorical notion of weak (co)limit that was introduced mainly for this purpose.  Similarly to the usual homotopy category, the homotopy $n$-category $\h_n\C$ of a (co)complete $\infty$-category $\C$ does not admit all small (co)limits in general, but it does admit \emph{better} (= higher) weak (co)limits as $n$ grows. The general properties of the homotopy $n$-categories suggest the following class of $n$-categories as a convenient context for refined versions of the GAFTs (see  \cite[Section 3]{Ra}). 

 \begin{definition*} [Definition \ref{weakly_cocomplete_def}]
Let $n \geq 1$ be an integer or $n=\infty$. A \emph{(finitely) weakly complete $n$-category} is an $n$-category $\C$ which  admits small (finite) products and weak pullbacks of order $(n-1)$. (There is an obvious dual notion of a (finitely) weakly cocomplete $n$-category.)
\end{definition*}

The definition and properties of higher weak (co)limits will be reviewed and developed further in \textbf{Section \ref{section:higher_weak_colimits}}.
We note that the homotopy $n$-category of a (finitely) complete $\infty$-category is a (finitely) weakly complete $n$-category. Moreover, a (finitely) weakly complete $\infty$-category is simply a (finitely) complete $\infty$-category. We emphasize here the special double role of $n$ in the definition: 
a $k$-category $\C$ is a (finitely) weakly complete $n$-category, for $n > k$, if and only if $\C$ is a (finitely) complete $k$-category. Thus, the definition essentially distinguishes the various categorical levels. 

\medskip

Our main adjoint functor theorems (Theorem \ref{nGAFTfin} and Theorem \ref{nGAFT}) generalize the corresponding GAFTs for $\infty$-categories in \cite[Theorems 3.2.5 and 3.2.6]{NRS} to (finitely) weakly complete $n$-categories. We refer to \textbf{Section \ref{AFTs}} and Definition \ref{solution_set_conditions} for the precise definitions of the $h$-initial object condition and of the solution set condition that appear in the statements below. 

\begin{gaftA}[$n$-GAFT$_{\mathrm{fin}}$ -- Theorem \ref{nGAFTfin}]
Let $G \colon \D \to \C$ be a functor between $n$-categories,  where $n \geq 1$ is an integer or $n=\infty$. Suppose that $\D$ is a finitely weakly complete $n$-category. Then $G$ admits a left adjoint if and only if $G$ preserves finite products, weak pullbacks of order $(n-1)$, and satisfies the $h$-initial object condition.
\end{gaftA}

\begin{gaftB}[$n$-GAFT -- Theorem \ref{nGAFT}]
Let $G \colon \D \to \C$ be a functor between $n$-categories, where $n \geq 2$ is an integer or $n=\infty$. Suppose that $\D$ is a locally small weakly complete $n$-category and that 
$\C$ is 2-locally small. Then $G$ admits a left adjoint if and only if $G$ preserves small products, weak pullbacks of order $(n-1)$, and satisfies the solution set condition.
\end{gaftB}

Let us emphasize here again the special double role of $n$ in these statements: the choice of $n$ both specifies the context for $\D$ and $\C$ and determines appropriate conditions for $G$; the combination of both functions of $n$ plays an important role in the proofs. As a consequence, we note that each of these two theorems states a separate assertion for each $n$. 

The GAFTs of \cite{NRS} -- as well as Freyd's GAFT -- are the special cases of the above statements for $n=\infty$. The general strategy for the proofs of $n$-GAFT$_{\mathrm{fin}}$ and $n$-GAFT is comparable to the strategy used for the proofs of GAFT$_{\mathrm{fin}}$ and GAFT in \cite{NRS}, but there are some interesting differences, too, since the proofs of these refined statements have a stronger homotopy-theoretic (or obstruction-theoretic) flavor. Since the property that a functor $G \colon \D \to \C$ admits a left adjoint is equivalent to the existence of initial objects $(c \to G(d))$ in the slice $\infty$-categories $G_{c/}$ for every $c \in \C$ (see Proposition \ref{characterization_adjoint_functor}), these theorems will be obtained as consequences of suitable criteria for the existence of initial objects -- Lemma \ref{criterionA} (Criterion A) and Lemma \ref{criterionB} (Criterion B), respectively.  These criteria and the proofs of the 
$n$-GAFTs are discussed in \textbf{Section \ref{AFTs}}. Moreover, as a consequence of the $n$-GAFTs, we also obtain the following result about detecting adjoint functors at the level of the (ordinary) homotopy category (this generalizes \cite[Theorem 3.3.1]{NRS}). 

\begin{gaftA}[Theorem \ref{homotopy_detect_adjoint}]
Let $\D$ be a finitely weakly complete $n$-category and $\C$ an $n$-category, where $n \geq 1$ is an integer or $n=\infty$. Let $G \colon \D \to \C$ be a functor which preserves finite products and weak pullbacks of order $(n-1)$. Then $G$ admits a left adjoint if and only if $\h(G) \colon \h(\D) \to \h(\C)$ admits a left adjoint. 
\end{gaftA}

Even though the $n$-GAFTs produce refinements of the GAFTs in \cite{NRS} for suitable $n$-categories which do not admit all small (co)limits, they still do not fully address the connection with
the AFTs that arise from Brown representability. Our second goal in this paper is to study a Brown representability context for higher (homotopy) 
categories and identify classes of (homotopy) $n$-categories which satisfy Brown representability -- this is done in \textbf{Section \ref{section:higher_Brown_rep}}. Brown representability for $n$-categories is defined in terms of the following $n$-categorical version of Brown's 
original conditions. We denote here by $\Sn$ the $n$-category of $(n-1)$-truncated objects in the $\infty$-category of (small) spaces $\S$.

\begin{definition*}[Definition \ref{Brown_rep_def}]
Let $\C$ be a locally small weakly cocomplete $n$-category, where $n \geq 1$ is an integer or $n=\infty$. We say that $\C$ \emph{satisfies Brown representability} if for any given functor $F \colon \C^{\op} \to \Sn$, the following holds: $F$ is representable  if (and only if) $F$ satisfies the conditions 
(B1)--(B2) below.  
\begin{enumerate}
\item[(B1).] For any small coproduct $\coprod_{i\in I} x_i$ in $\C$, the canonical morphism in $\Sn$
$$
F\left(\coprod_{i\in I} x_i\right)\longrightarrow \prod_{i\in I} F(x_i)
$$
is an equivalence.
\item[(B2).] For every weak pushout in $\C$ of order $(n-1)$
\begin{equation*} 
\xymatrix{
x \ar[r] \ar[d] & y \ar[d] \\ 
z \ar[r] & w}
\end{equation*}
the canonical morphism in $\Sn$
$$
F(w)\longrightarrow F(y)\times_{F(x)}F(z)
$$
is $(n-1)$-connected. 
\end{enumerate}
\end{definition*}

This definition restricts for $n=1$ to the familiar Brown representability context for ordinary categories (see, for example, \cite{He, NRS}). An analogous $2$-categorical Brown representability context was also considered recently in \cite{Ca}.  Let us emphasize again that this definition singles out a class of $n$-categories which is specific to each $n$. Similarly to classical Brown representability, the property that an $n$-category $\C$ satisfies Brown representability is closely related to an AFT for $\C$. This connection is explained in the following proposition. 

\begin{proposition*}[Proposition \ref{Brown_implies_adjoints}]
Let $\C$ and $\D$ be locally small $n$-categories, where $n \geq 1$ is an integer or $n=\infty$. Suppose that $\C$ is a weakly cocomplete $n$-category and satisfies Brown representability. Then a functor $F \colon \C \to \D$ admits a right adjoint if and only if $F$ satisfies the following properties:
\begin{itemize}
\item[(B1$'$).] $F$ preserves small coproducts. 
\item[(B2$'$).] $F$ preserves weak pushouts of order $(n-1)$. 
\end{itemize}
\end{proposition*}

We introduce in \textbf{Section \ref{section:higher_Brown_rep}} a class of weakly cocomplete $n$-categories, called \emph{compactly generated $n$-categories}, and prove the following Brown representability theorem in the context of $n$-categories. 

\begin{gaftA} [Theorem \ref{higher_Brown}]
Let $\C$ be a compactly generated $n$-category, where $n \geq 1$ is an integer or $n=\infty$. Then $\C$ satisfies Brown representability.
\end{gaftA}

The notion of a compactly generated $n$-category contains some subtleties; we refer to Subsection \ref{subsec:compactly_generated} for the precise definition. Examples include the homotopy $n$-category of a finitely presentable $\infty$-category if $n \geq 2$, and the homotopy $n$-category of a finitely presentable stable $\infty$-category for any $n \geq 1$. 
The strategy for the proof of this Brown representability theorem for $n$-categories is comparable to Brown's original proof method, but various types of refinements of this method are required in order to apply to the $n$-categorical context; these make essential use of higher weak colimits in combination with the properties (B1)--(B2). The representability of a functor $F \colon \C^{\op} \to \Sn$ is equivalent to the existence of an initial object (or ``universal element'') $(* \to F(c))$ in $F_{*/}$. For the existence of an initial object in this context, we formulate a third general criterion for initial (or terminal) objects which applies specifically to $n$-categories with a set of compact objects that jointly detect equivalences -- see Lemma \ref{criterion_C} (Criterion C).  

Let us also remark that our previous results on Brown representability for $\infty$-categories in \cite{NRS} are essentially special cases of the theorem above for $n=1$. The purpose of the results in \cite[Section 5]{NRS} was indeed to identify classes of $\infty$-categories whose (ordinary) homotopy categories satisfy Brown representability. We refer to the relevant remarks in Section \ref{section:higher_Brown_rep} for more detailed explanations. 

\smallskip

Since the class of $n$-categories which satisfy Brown representability is closed under localizations (Proposition \ref{local_preserves_Brown}), we obtain the following result as a corollary.  

\begin{corollary*}[Corollaries \ref{higher_Brown_2} and \ref{higher_Brown_3}] Let $\C$ be a presentable $\infty$-category and let $\D$ be a locally small $n$-category, where $n \geq 1$ is an integer or $n = \infty$.
\begin{enumerate}
\item Suppose that $\C$ is stable. Then $\h_n\C$ satisfies Brown representability. As a consequence, a functor $F \colon \h_n \C \to \D$ admits a right adjoint if and only if $F$ preserves small coproducts and weak pushouts of order $(n-1)$.
\item Suppose that $n \geq 2$. Then $\h_n\C$ satisfies Brown representability.  As a consequence, a functor $F \colon \h_n \C \to \D$ admits a right adjoint if and only if $F$ preserves small coproducts and weak pushouts of order $(n-1)$.
\end{enumerate}
\end{corollary*}

Note that this adjoint functor theorem for $n=\infty$ recovers the left adjoint functor theorem for presentable $\infty$-categories \cite[Corollary 5.5.2.9(1)]{HTT}, \cite[Section 4]{NRS}. Thus, we obtain a collection of Brown representability theorems for $\h_n\C$, $1 \leq n \leq \infty$, which bridges the gap between the classical Brown representability theorems (e.g., for suitable triangulated categories) and the left adjoint functor theorem for presentable $\infty$-categories.

\medskip

\noindent \textbf{Conventions and terminology.} As in \cite{NRS}, we work in a model $\mathbb{V}$ of $\mathrm{ZFC}$-set theory which contains an inaccessible cardinal. We use the associated Grothendieck universe $\mathbb U \in \mathbb{V}$ to distinguish between small and large sets. More specifically, a set is called \emph{small} if it belongs to $\mathbb U$. Our results do not depend on these set-theoretical assumptions in any essential way; these are used as a convenient and standard convention.  

A simplicial set is a functor $K \colon \Delta^{\op}\to \Set_{\mathbb V}$. A simplicial set $K \colon \Delta^{\op} \to \Set_{\mathbb V}$ is \emph{small} if $K_n \in \mathbb U$ for each $[n]\in \Delta^{\op}$. An $\infty$-category (= quasi-category) is \emph{essentially small} if it is (Joyal) equivalent to a small simplicial set. 
An $\infty$-category $\C$ is called \emph{locally small} if for every small set $S$ of objects in $\C$, the full subcategory of $\C$ spanned by $S$ is essentially small (see \cite[5.4.1]{HTT}). 

An $\infty$-category is (finitely) \emph{complete} (resp. \emph{cocomplete})  if it admits all limits (resp. colimits) indexed by small (finite) simplicial sets. For a simplicial set $K$, we will often use the notation $K^{\triangleright}$ for $K \ast \Delta^{0}$ and $K^{\triangleleft}$ for $\Delta^0 \ast K$.

\smallskip

For $n \geq 0$, a map $f \colon X \to Y$ between spaces (= Kan complexes) is called \emph{$n$-connected} if $\pi_0(f)$ is surjective and for every $x \in X$ the induced morphism $\pi_k(f, x)$ is an isomorphism if $k < n$ and an epimorphism if $k = n$ -- this is the classical convention. Every map is $(-1)$-connected. A space $X$ is $n$-connected if the map $(X \to *)$ is $(n+1)$-connected. More generally, a simplicial set $X$ is called \emph{$n$-connected} if it is weakly equivalent to an $n$-connected Kan complex. A space $X$ is \emph{$n$-truncated} if $\pi_k(X, x) = 0$ for all $x \in X$ and $k > n$. 

\smallskip

We will use the term $n$-category in the sense of \cite[2.3.4]{HTT}; this is a model for $(n,1)$-categories in the context of quasi-categories. For an $\infty$-category $\C$, we denote its homotopy $n$-category by $\h_n\C$. We will usually use the standard notation $\h(\C)$ to denote the usual homotopy category when $n=1$. $\S$ will denote the $\infty$-category of (small) spaces and $\Sn \subset \S$ the full subcategory of $(n-1)$-truncated spaces. $\Sn$ is equivalent to an $n$-category.

\medskip

\noindent \textbf{Acknowledgements.} Hoang Kim Nguyen and George Raptis gratefully acknowledge the support of the \emph{SFB 1085 -- Higher Invariants} (University of Regensburg) funded by the DFG. The authors thank the anonymous referee for their useful comments. 

\section{Higher weak (co)limits} \label{section:higher_weak_colimits}

\subsection{Recollections} Higher weak (co)limits are simultaneously a higher categorical generalization of ordinary weak (co)limits and a weakening of the notion of (co)limits in higher categories.  We review the definition of higher weak (co)limits and some of their basic properties from \cite[Section 3]{Ra}.  

First we recall that a space (= Kan complex) $X$ is $k$-\emph{connected}, for some $k \geq -1$, if it is non-empty and 
$\pi_i(X, x) \cong 0$ for every $x \in X$ and $i \leq k$. Every space $X$ is $(-2)$-connected. A space $X$ is $(-1)$-connected 
(resp. $0$-connected, $\infty$-connected) if it is non-empty (resp. connected, contractible). 

\begin{definition} Let $\C$ be an $\infty$-category and let $t \geq -1$ be an integer or $t = \infty$. 
\begin{itemize}
\item[(1)] An object $x \in \C$ is called \emph{weakly initial of order} $t$ if the mapping space $\map_{\C}(x, y)$ is $(t-1)$-connected for every object $y \in \C$. 
\item[(2)] An object $x \in \C$ is called \emph{weakly teminal of order} $t$ if the mapping space $\map_{\C}(y, x)$ is $(t-1)$-connected for every object $y \in \C$. 
\end{itemize}
\end{definition}

\begin{proposition} \label{Kan_complex_weakly_initial}
Let $\C$ be an $\infty$-category and let $t > 0$ be an integer or $t=\infty$. The full subcategory $\C'$ of $\C$ which is spanned by the weakly initial (resp. weakly terminal) objects of order $t$ is either empty or a $t$-connected $\infty$-groupoid. 

In particular, any two weakly initial (resp. weakly terminal) objects of order $t > 0$
are equivalent. (This fails in general for $t=-1, 0$.) 
\end{proposition}
\begin{proof}
See \cite[Proposition 3.4, Remarks 3.5 and 3.8]{Ra}.
\end{proof}

\begin{lemma} \label{characterization_n-connected_left/right}
Let $p\colon \C\to \D$ be a left or right fibration and let $t \geq -1$ be an integer or $t=\infty$. Then $p$ has $(t-1)$-connected fibers if and only if for every $0\leq k\leq t$, every lifting problem of the form
\[
\begin{tikzcd}[column sep=small]
\partial \Delta^k  \arrow[r] \arrow[d,hook] & \C \arrow[d, "p"] \\
\Delta^k   \arrow[ur,dashed]  \arrow[r] & \D \\
\end{tikzcd}
\]
admits a diagonal filler.
\end{lemma}

\begin{proof}
This generalizes \cite[Lemma 2.1.3.4]{HTT} which treated the case $t = \infty$. The same proof as in \cite[Lemma 2.1.3.4]{HTT} applies to the more general case. 
\end{proof}

\begin{proposition} \label{characterization_weakly_initial/terminal}
Let $\C$ be an $\infty$-category, $x \in \C$ an object, and let $t \geq -1$ be an integer or $t = \infty$. We denote by $p \colon \C_{x/} \to \C$ (resp. $q \colon \C_{/x} \to \C$) the associated left (resp. right) fibration. Then the following hold:
\begin{enumerate}
\item $x \in \C$ is weakly initial of order $t$ if and only if every lifting problem, where $0 \leq k \leq t$, 
\[
\begin{tikzcd}[column sep=small]
\partial \Delta^k  \arrow[r] \arrow[d,hook] & \C_{x/} \arrow[d, "p"] \\
\Delta^k   \arrow[ur,dashed]  \arrow[r] & \C \\
\end{tikzcd}
\]
admits a diagonal filler. 
\item $x \in \C$ is weakly terminal of order $t$ if and only if every lifting problem, where $0 \leq k \leq t$, 
\[
\begin{tikzcd}[column sep=small]
\partial \Delta^k  \arrow[r] \arrow[d,hook] & \C_{/x} \arrow[d, "q"] \\
\Delta^k   \arrow[ur,dashed]  \arrow[r] & \C \\
\end{tikzcd}
\]
admits a diagonal filler. 
\end{enumerate} 
\end{proposition}
\begin{proof}
This follows directly from Lemma \ref{characterization_n-connected_left/right}.
\end{proof}

\begin{definition}
Let $\C$ be an $\infty$-category, $K$ a simplicial set, and let $f_0 \colon K \to \C$ be a $K$-diagram in $\C$. 
\begin{itemize}
\item[(1)] A weakly initial object $f \in \C_{f_0/}$ of order $t$ is called a \emph{weak colimit of $f_0$ of order} $t$.   
\item[(2)] A weakly terminal object $f \in \C_{/f_0}$ of order $t$ is called a \emph{weak limit of $f_0$ of order} $t$.   
\end{itemize}
\end{definition}

The following proposition gives several equivalent characterizations of higher weak (co)limits. 

\begin{proposition} \label{characterization_weak_colimit}
Let $\C$ be an $\infty$-category and let $t \geq -1$ be an integer or $t = \infty$. Let $K$ be a simplicial set and let $f \colon K^{\triangleright} \to \C$ be a cone on $f_0 = f_{|K} \colon K \to \C$ with cone object $x \in \C$. We denote by $p \colon \C_{f/} \to \C_{f_0/}$ the associated left fibration. Then the following are equivalent: 
\begin{enumerate} 
\item[(a)] $f$ is a weak colimit of $f_0$ of order $t$.
\item[(b)] The fibers of $p \colon \C_{f/} \to \C_{f_0/}$ are $(t-1)$-connected. 
\item[(c)]  For every object $y \in \C$, the canonical restriction map
$$\map_{\C}(x, y) \simeq \map_{\C^{K^{\triangleright}}}(f, c_y) \to \map_{\C^K}(f_0, c_y)$$
is $t$-connected. ($c_y$ denotes respectively the constant diagram at $y \in \C$.)
\item[(d)] For $0 \leq k \leq t$, every lifting problem
\[
\begin{tikzcd}[column sep=small]
\partial \Delta^k  \arrow[r] \arrow[d,hook] & \C_{f/} \arrow[d, "p"] \\
\Delta^k   \arrow[ur,dashed]  \arrow[r] & \C_{f_0/} \\
\end{tikzcd}
\]
admits a diagonal filler.  
\end{enumerate}
An analogous statement holds also for weak limits: Suppose that $f \colon K^{\triangleleft} \to \C$ is a cone on 
$f_0 = f_{|K} \colon K \to \C$ with cone object $x \in \C$ and let $q \colon \C_{/f} \to \C_{/f_0}$ denote the associated right fibration. 
Then the following are equivalent: 
\begin{enumerate} 
\item[(a)]  $f$ is a weak limit of $f_0$ of order $t$.
\item[(b)] The fibers of $q \colon \C_{/f} \to \C_{/f_0}$ are $(t-1)$-connected. 
\item[(c)]  For every object $y \in \C$, the canonical restriction map
$$\map_{\C}(y, x) \simeq \map_{\C^{K^{\triangleleft}}}(c_y, f) \to \map_{\C^K}(c_y, f_0)$$
is $t$-connected. ($c_y$ denotes respectively the constant diagram at $y \in \C$.)
\item[(d)] For $0 \leq k \leq t$, every lifting problem
\[
\begin{tikzcd}[column sep=small]
\partial \Delta^k  \arrow[r] \arrow[d,hook] & \C_{/f} \arrow[d, "q"] \\
\Delta^k   \arrow[ur,dashed]  \arrow[r] & \C_{/f_0} \\
\end{tikzcd}
\]
admits a diagonal filler.
\end{enumerate}
\end{proposition}
\begin{proof}
(a) $\Leftrightarrow$ (b) is obvious. For (a) $\Leftrightarrow$ (c), see \cite[Proposition 3.9]{Ra}. (b) $\Leftrightarrow$ (d) is an easy consequence of Lemma \ref{characterization_n-connected_left/right} (cf. Proposition \ref{characterization_weakly_initial/terminal}).  
\end{proof} 

\begin{example} \label{weakly-colimits--1}
Every object $x \in \C$ is weakly initial (resp. weakly terminal) of order $(-1)$. More generally, any cone $f \colon K^{\triangleleft} \to \C$ (resp. $f \colon K^{\triangleright} \to \C$) on a $K$-diagram $f_0 = f_{|K} \colon K \to \C$ is a weak limit (resp. weak colimit) of $f_0$ of order $(-1)$. 
\end{example}

\begin{example} \label{weak-colimits-0}
Suppose that $\C$ is (the nerve of) an ordinary category. An object $x \in \C$ is weakly initial (resp. weakly terminal) of order $0$ if and only if $x$ is weakly initial (resp. weakly terminal) in the usual sense. Similarly, a weak (co)limit in $\C$ of order $0$ is exactly a weak (co)limit in the usual sense. 
\end{example}

\begin{example} \label{weak-colimits-infinity}
An object $x \in \C$ is weakly initial (resp. weakly terminal) of order $\infty$ if and only if $x$ is initial (resp. terminal). A cone $f \colon K^{\triangleleft} \to \C$ (resp. $f \colon K^{\triangleright} \to \C$) on a $K$-diagram $f_0 = f_{|K} \colon K \to \C$ is a weak limit (resp. weak colimit) of $f_0$ of order $\infty$ if and only if $f$ is a limit (resp. colimit) of $f_0$.   
\end{example}

\begin{remark} \label{weak-colimits-n-cat-remark}
The last example has the following useful variation; for simplicity, we formulate this only for weak colimits. Suppose that $\C$ is an $n$-category (see \cite[2.3.4]{HTT}). If $f \colon K^{\triangleright} \to \C$ is a weak colimit of $f_0 = f_{|K} \colon K \to \C$ of order $t \geq n$, then $f$ is a colimit diagram. This is because $\C_{f_0/}$ is an $n$-category (see \cite[Corollary 2.3.4.10]{HTT}) and the mapping spaces in an $n$-category are $(n-1)$-truncated.  
\end{remark}

\subsection{Weakly (co)complete $n$-categories} \label{weakly_cocomplete_n-categories} Weakly (co)complete $n$-categories determine a class of $n$-categories which lies between weakly (co)complete (ordinary) categories and (co)complete $\infty$-categories. 

\begin{definition} \label{weakly_cocomplete_def} 
Let $\C$ be an $\infty$-category and let $n \geq 1$ be an integer or $n=\infty$. 
\begin{enumerate}
\item $\C$ is a \emph{(finitely) weakly cocomplete $n$-category} if $\C$ is an $n$-category which admits small (finite) coproducts and weak pushouts of order $(n-1)$. In this case, we also say that $\C$ is (finitely) weakly $n$-cocomplete. 
\item $\C$ is a \emph{(finitely) weakly complete $n$-category} if $\C$ is an $n$-category which  admits small (finite) products and weak pullbacks of order $(n-1)$. In this case, we also say that $\C$ is (finitely) weakly $n$-complete. 
\end{enumerate}
\end{definition}

Note that a (finitely) weakly cocomplete $\infty$-category is (finitely) cocomplete and every (finitely) cocomplete $n$-category is also (finitely) weakly $n$-cocomplete. 

While it is convenient to state the definition in terms of (co)products and pushouts/pullbacks, weakly (co)complete $n$-categories admit also further weak (co)limits of variable orders. The following proposition explains more generally what other higher weak (co)limits can be deduced from the existence of (co)products and higher weak pullbacks/pushouts. 

\begin{proposition} \label{weakly_cocomplete_colimits}
Let $\C$ be an $\infty$-category which admits small (finite) products (resp. coproducts) and weak pullbacks (resp. weak pushouts) of order $t$. Then for every small (finite) simplicial set $K$ of dimension $d \leq t + 2$,  $\C$ admits weak $K$-limits (resp. weak $K$-colimits) of order $(t-d+1)$. 
\end{proposition}
\begin{proof}
This is shown by induction on the dimension $d$ of the simplicial set $K$ using \cite[Proposition 3.10]{Ra}. We sketch the details for completeness. First, the claim is obvious when $K$ is empty and for $d = 0$ (and any $t$). Assume by induction that the claim holds when the dimension is less than $d$ and let $K \to \C$ be a $K$-diagram indexed by a $d$-dimensional small (finite) simplicial set $K$, where $0 < d \leq t+2$. We have a pushout diagram 
$$
\xymatrix{
\bigsqcup_I \partial \Delta^d \ar[r] \ar[d] & \sk_{d-1}(K) \ar[d] \\
\bigsqcup_I \Delta^d \ar[r] & K
}
$$
where $I$ is the (small/finite) index set of the non-degenerate $d$-simplices of $K$. By the inductive assumption, the respective restrictions of $K \to \C$ to $\sk_{d-1}(K)$ and $\bigsqcup_I \partial \Delta^d$ admit weak limits of order $(t-(d-1)+1) \geq 0$. Moreover, the composite diagram $\bigsqcup_I \Delta^d \to K \to \C$ admits a limit cone; this is given by the product of the limits coming from each $I$-component $\Delta^d \to \C$, which exist because $\Delta^d$ has an initial object. Then the result follows from \cite[Proposition 3.10]{Ra} 
\end{proof}

The main motivation for the definition of (finitely) weakly cocomplete (resp. complete) $n$-categories comes from the following class of examples. 

\begin{example} \label{weakly_cocomplete_example}
Let $\C$ be a (finitely) cocomplete (resp. complete) $\infty$-category. Then the homotopy $n$-category $\h_n \C$ is (finitely) weakly $n$-cocomplete (resp. $n$-complete) \cite[Proposition 3.20, Corollary 3.22]{Ra}. 
\end{example}
 
We refer to \cite[Sections 3 and 6.1--6.2]{Ra} for more details about the properties of higher weak (co)limits in higher homotopy categories. 

\begin{remark}
It is important to observe the role of $n$ in Definition \ref{weakly_cocomplete_def} and how this definition singles out a distinguished class of $n$-categories which is specific to each $n \geq 1$. An $n$-category can always be regarded as an $(n+1)$-category, but a weakly (finitely) cocomplete $n$-category $\C$ is not a weakly (finitely) cocomplete $(n+1)$-category in general -- this happens only when $\C$ is (finitely) cocomplete. 
\end{remark}

\subsection{Higher weak (co)limits and slice $\infty$-categories} It is generally known how to identify appropriate (co)limits in slice $\infty$-categories in 
terms of (co)limits in the underlying $\infty$-category. The purpose of this subsection is to establish analogous inheritance properties of higher weak colimits 
under passing to appropriate slice $\infty$-categories. 

\begin{proposition}\label{limits under diagrams}
Let $\C$ be an $\infty$-category and let $t \geq 0$ be an integer or $t = \infty$. Let $K$ be a $d$-dimensional simplicial set and $f  \colon K \to \C$ a $K$-diagram in $\C$. We denote by $p \colon \C_{f/} \to \C$ (resp. $q \colon \C_{/f} \to \C$) the associated left (resp. right) fibration. 

Let $L$ be a simplicial set. Then the following statements hold:
\begin{enumerate}
\item Assume that $d \leq t$ and suppose that $g \colon L^{\triangleleft} \to \C_{f/}$ is a cone on $g_0 = g_{|L} \colon L \to \C_{f/}$ such that $p \circ g \colon L^{\triangleleft} \to \C$ is a weak limit of $p \circ g_0$ of order $t$. Then $g$ is a weak limit of order $(t-d-1)$.
\item Assume that $d \leq t$ and let $g_0 \colon L \to \C_{f/}$ be an $L$-diagram. Suppose that $\overline{g} \colon L^{\triangleleft} \to \C$ is a weak limit of order $t$ of the composition $L \stackrel{g_0}{\to} \C_{f/} \stackrel{p}{\to} \C$. Then we can lift $\overline{g}$ to a weak limit of $g_0$ in $\C_{f/}$ of order $(t-d-1)$. 
\item Suppose that $g \colon L^{\triangleleft}\to \C_{f/}$ is a weak limit of order $t$ of an $L$-diagram $g_0 = g_{|L} \colon L \to \C_{f/}$, where $t \geq 1$. Assume that the diagram $L \stackrel{g_0}{\to} \C_{f/} \stackrel{p}{\to} \C$ admits a weak limit of order $(t+d+1)$. Then this weak limit is given (up to equivalence) by the composition $L^{\triangleleft} \stackrel{g}{\to} \C_{f/} \stackrel{p}{\to} \C$.
\end{enumerate}

Analogous statements also hold for weak colimits in $\C_{/f}$:

\begin{enumerate}
\item Assume that $d \leq t$ and suppose that $g \colon L^{\triangleright} \to \C_{/f}$ is a cone on $g_0 = g_{|L} \colon L \to \C_{/f}$ such that $q \circ g \colon L^{\triangleright} \to \C$ is a weak colimit of $q \circ g_0$ of order $t$. Then $g$ is a weak colimit of order $(t-d-1)$.
\item Assume that $d \leq t$ and let $g_0 \colon L \to \C_{/f}$ be an $L$-diagram. Suppose that $\overline{g} \colon L^{\triangleright} \to \C$ is a weak colimit  of order $t$ of the composition $L \stackrel{g_0}{\to} \C_{/f} \stackrel{q}{\to} \C$. Then we can lift $\overline{g}$ to a weak colimit of $g_0$ in $\C_{/f}$ of order $(t-d-1)$. 
\item Suppose that $g \colon L^{\triangleright}\to \C_{/f}$ is a weak colimit of order $t$ of an $L$-diagram $g_0 = g_{|L} \colon L \to \C_{/f}$, where $t \geq 1$. Asume that the diagram $L \stackrel{g_0}{\to} \C_{/f} \stackrel{q}{\to} \C$ admits a weak colimit of order $(t + d + 1)$. Then this weak colimit is given (up to equivalence) by the composition $L^{\triangleright} \stackrel{g}{\to} \C_{/f} \stackrel{q}{\to} \C$.
\end{enumerate}
\end{proposition}
\begin{proof}(1): By Proposition \ref{characterization_weak_colimit}, it suffices to prove that for every diagram as follows, where $0 \leq k \leq t- d -1$,  
\[
\begin{tikzcd}[column sep=small]
\partial \Delta^k  \arrow[r] \arrow[d,hook] & (\C_{f/})_{/g} \arrow[d] \\
\Delta^k   \arrow[ur,dashed]  \arrow[r] & (\C_{f/})_{/g_0} \\
\end{tikzcd}
\]
there is a diagonal filler which makes the diagram commutative. Equivalently, it suffices to prove that the associated adjoint lifting problem
\[
\begin{tikzcd}[column sep=small]
K \ast \partial \Delta^k \arrow[r] \arrow[d,hook] & \C_{/p \circ g} \arrow[d] \\
K \ast \Delta^k  \arrow[ur,dashed] \arrow[r] & \C_{/p \circ g_0} \\
\end{tikzcd}
\]
admits a diagonal filler. Note that the dimension of $K \ast  \Delta^k$ is $\leq t$. Then Proposition \ref{characterization_weak_colimit}(d) shows that a diagonal filler exists since $p\circ g$ is a weak limit of order $t$ in $\C$ by assumption. 

(2): Using (1), we only need to show that we can lift $\overline{g}$ along $p \colon \C_{f/} \to \C$ to a cone on $g_0$. By adjunction, this amounts to finding a diagonal filler that makes the following diagram commutative:
\[
\begin{tikzcd}
\varnothing \arrow[d,hook] \arrow[r] & \C_{/\overline{g}} \arrow[d] \\
K \arrow[ur,dashed] \arrow[r] & \C_{/p \circ g_0}.
\end{tikzcd}
\]
Here the bottom morphism is defined by $g_0$ and the right vertical map is the canonical right fibration (note that $p \circ g_0 = \overline{g}_{|L}$). 
Since $\overline{g}$ is a weak limit of order $t$ and the dimension $d$ of $K$ is $\leq t$, this diagonal filler exists by Proposition \ref{characterization_weak_colimit}(d). 

(3): 
Let $\overline{h} \colon L^{\triangleleft} \to \C$ be a weak limit of $L \stackrel{g_0}{\to} \C_{f/} \stackrel{p}{\to} \C$ of order $(t+d+1)$. By (2), we may lift this along $p$ to a weak limit $h \colon L^{\triangleleft} \to \C_{f/}$ of order $t$. The two weak limits $h$ and $g$ of $g_0 \colon L \to \C_{f/}$ are of order $\geq 1$ by assumption. Thus, they must be equivalent by Proposition \ref{Kan_complex_weakly_initial}. As a consequence, the cone $\overline{h}$ is also equivalent to the image of the weak limit $g$ under $p$, that is, the composition $L^{\triangleleft} \stackrel{g}{\to} \C_{f/} \stackrel{p}{\to} \C$.
\end{proof}

\begin{corollary}\label{limits_under_objects}
Let $\C$ be an $\infty$-category, $x \in \C$ an object, and let $t \geq 0$ be an integer or $t = \infty$. We denote by $p \colon \C_{x/} \to \C$ (resp. $q \colon \C_{/x} \to \C$) the associated left (resp. right) fibration.  
\begin{enumerate}
\item Let $L$ be a simplicial set and $g_0 \colon L \to \C_{x/}$ an $L$-diagram in $\C_{x/}$. Then a weak limit of $p \circ g_0 \colon L \to \C$ of order $t$ lifts to a weak limit of $g_0$ of order $(t-1)$.
\item Let $L$ be a simplicial set and $g_0 \colon L \to \C_{/x}$ an $L$-diagram in $\C_{/x}$. Then a weak colimit of $q \circ g_0 \colon L \to \C$ of order $t$ lifts to a weak colimit of $g_0$ of order $(t-1)$.
\end{enumerate} 
\end{corollary}

\begin{remark} \label{best_possible_estimate1}
Consider the following special case (for $L=\varnothing$): If $x \in \C$ is weakly initial of order $t$, then $(x \stackrel{\mathrm{id}}{\to} x) \in \C_{/x}$ is weakly initial of order $(t-1)$ -- this can easily be deduced directly from the definition of a weakly initial object. This special case shows that the estimates of Corollary \ref{limits_under_objects} are best possible in general. For example, if $\C$ is (the nerve of) an ordinary category and $x \in \C$ is weakly initial (of order $0$), then 
$(x \stackrel{\mathrm{id}}{\to} x) \in \C_{/x}$ is not weakly initial (of order 0) in general (this happens, for example, in the case where $\C$ is the category of non-empty sets and $x$ is a set with two elements).    
\end{remark}

\subsection{Higher weak (co)limits and pullbacks of $\infty$-categories} \label{weak_colimits_slices} It is well known how (co)limits in pullbacks of $\infty$-categories (along (co)limit--preserving functors) are related and induced from (co)limits in the corresponding $\infty$-categories (see, for example, \cite[Lemma 5.4.5.5]{HTT}). The following proposition establishes an analogous property for higher weak (co)limits in pullbacks of $\infty$-categories.     

\begin{proposition}\label{weak_limits_in_pullbacks}
Let $K$ be a simplicial set and let
\[
\begin{tikzcd}
\A \arrow[d,"\gamma"'] \arrow[r,"\beta"] & \B \arrow[d, "\psi"] \\
\C \arrow[r,"\phi"'] & \D
\end{tikzcd}
\]
be a homotopy pullback of $\infty$-categories (with respect to the Joyal model structure). Suppose that $\B$ and $\C$ admit weak $K$-limits of orders $t_{\B}$ and $t_{\C}$ respectively, and $\phi$ sends weak $K$-limits of order $t_{\C}$ to weak $K$-limits of order $t_{\D}$. 

Let $t : = \mathrm{min}(t_{\B}, t_{\C}, t_{\D} - 1)$ and let $H_0 \colon K \to \A$ be a $K$-diagram in $\A$. 
\begin{enumerate}
\item A cone $H \colon K^{\triangleleft} \to \A$ is a weak limit of $H_0$ of order $t$ if $\beta \circ H \colon K^{\triangleleft} \to \B$ is a weak limit of $\beta \circ H_0$ of order $t_{\B}$ and $\gamma \circ H \colon K^{\triangleleft} \to \C$ is a weak limit of $\gamma \circ H_0$ of order $t_{\C}$. The converse also holds if $t \geq 1$ and $\psi$ sends weak $K$-limits of order $t_{\B}$ to weak $K$-limits of order $t_{\D}$. 
\item Suppose that $t_{\D} \geq 1$ and $\psi$ sends weak $K$-limits of order $t_{\B}$ to weak $K$-limits of order $t_{\D}$. Then $\A$ admits weak $K$-limits of order $t$. 
\end{enumerate}
An analogous dual statement holds also for weak colimits in $\A$. 
\end{proposition}
\begin{proof} We will apply the characterization of higher weak (co)limits from Proposition \ref{characterization_weak_colimit}(c).  Let $x \in \A$ be the cone object of the cone $H$ and let $a \in \A$ be an arbitrary object. Then we have the following diagram in the $\infty$-category of spaces, in which the two horizontal squares are pullback squares:
\[
\footnotesize{
\begin{tikzcd}[row sep=tiny, column sep=tiny]
& \map_{\A}(a,x) \arrow[dd,dashrightarrow] \arrow[ld, rightarrow] \arrow[rr, rightarrow] && \map_{\B}(\beta(a),\beta(x)) \arrow[dd] \arrow[dl] \\
\map_{\C}(\gamma(a),\gamma(x)) \arrow[dd] \arrow[rr] && \map_{\D}(\phi(\gamma(a)),\phi(\gamma(x))) \arrow[dd] \\
& \map_{\A^K}(c_a, H_0) \arrow[dl,dashrightarrow] \arrow[rr,dashrightarrow] && \map_{\B^K}(c_{\beta(a)}, \beta \circ H_0) \arrow[dl] \\
\map_{\C^K}(c_{\gamma(a)}, \gamma \circ H_0) \arrow[rr] && \map_{\D^K}(c_{\phi(\gamma(a))}, (\phi \circ \gamma) \circ H_0 ).
\end{tikzcd}
}\]
The vertical maps are the canonical maps from Proposition \ref{characterization_weak_colimit}(c). The bottom square is a pullback because $\A^K$ is a homotopy pullback of the induced diagram of $\infty$-categories: $(\C^K \rightarrow \D^K \leftarrow \B^K)$. By assumption and Proposition \ref{characterization_weak_colimit}, the three front/solid vertical maps are $t_{\C}$-connected, $t_{\D}$-connected, and $t_{\B}$-connected, respectively. Contemplating the long exact sequence of homotopy groups (for pullbacks of spaces)
shows that the last (dashed) vertical map is $t$-connected. This proves the first part of (1). 

For the converse in (1), suppose that $H \colon K^{\triangleleft} \to \A$ is a weak $K$-limit of $H_0$ of order $t$. Since $t \geq 1$, by assumption, it follows that $t_{\B}, t_{\C}, t_{\D} \geq 1$. Let $H' \colon K^{\triangleleft} \to \B$ be a weak $K$-limit of $\beta \circ H_0$ of order $t_{\B}$ and 
$H'' \colon K^{\triangleleft} \to \C$ a weak $K$-limit of $\gamma \circ H_0$ of order $t_{\C}$. By assumption, the cones $\psi \circ H'$ and $\phi \circ H''$ are weak $K$-limits of order $\geq 1$, and therefore, by Proposition \ref{Kan_complex_weakly_initial}, they are equivalent. Then $H', H''$ and 
$\psi \circ H' \simeq \phi \circ H''$ induce a cone $\widetilde{H} \colon K^{\triangleleft} \to \A$ on $H_0$. Using (1), $\widetilde{H}$ is a weak $K$-limit of 
$H_0$ of order $t$. Since $t \geq 1$, Proposition \ref{Kan_complex_weakly_initial} implies that the cones $H$ and $\widetilde{H}$ are equivalent and then the desired result follows. 

For (2), suppose that $H_0 \colon K \to \A$ is a $K$-diagram in $\A$. Proceeding as in the previous argument above, we may extend this to a cone $H \colon K^{\triangleleft} \to \A$ such that $\beta \circ H$ is a weak $K$-limit of $\beta \circ H_0$ of order $t_{\B}$ and $\gamma \circ H$ is a weak $K$-limit of $\gamma \circ H_0$ of order $t_{\C}$. Then it follows from (1) that $H \colon K^{\triangleleft} \to \A$ is a weak $K$-limit of $H_0$ of order $t$, proving (2). 
\end{proof}

\begin{remark}\label{remark_weak_limits_in_pullbacks}
In the proof of the first part of Proposition \ref{weak_limits_in_pullbacks}(1), we do not need the assumption that $\mathscr{B}$ and $\C$ admit general weak $K$-limits (of orders $t_{\mathscr{B}}$ and $t_{\C}$).
\end{remark}

\begin{corollary} \label{limits_slices_under_objects}
Let $K$ be a simplicial set and let $\D$ be an $\infty$-category which admits weak $K$-limits of order $t$, where $t \geq 0$ is an integer or $t = \infty$. Suppose that $G:\D \to \C$ is a functor that preserves weak $K$-limits of order $t$. Then the $\infty$-category $G_{x/}$ has weak $K$-limits of order $(t-1)$ for any object $x \in \C$. (There is an analogous dual statement for weak colimits in $G_{/x}$.)
\end{corollary}
\begin{proof}
We recall that the $\infty$-category $G_{x/}$ is defined by the (homotopy) pullback:
\[
\begin{tikzcd}
G_{x/} \arrow[d, "q"] \arrow[r] & \C_{x/} \arrow[d, "p"] \\
\D \arrow[r,"G"'] & \C.
\end{tikzcd}
\]
Let $H_0:K \to G_{x/}$ be a $K$-diagram. By assumption, the induced diagram $q \circ H_0$ in $\D$ has a weak limit $H' \colon K^{\triangleleft} \to \D$ of order $t$. Since $G$ preserves weak $K$-limits of order $t$, the cone $G \circ H' \colon K^{\triangleleft} \to \C$ is again a weak $K$-limit of order $t$. By Corollary \ref{limits_under_objects}, there is a lift of $G \circ H'$ to a weak $K$-limit $H'' \colon K^{\triangleleft} \to \C_{x/}$ of order $(t-1)$. The cones 
$H'$ and $H''$ induce a cone $H \colon K^{\triangleleft} \to G_{x/}$ on $H_0$. Finally, by applying  Proposition \ref{weak_limits_in_pullbacks} (and Remark \ref{remark_weak_limits_in_pullbacks}), we conclude that the induced cone $H \colon K^{\triangleleft} \to G_{x/}$ is a weak $K$-limit of $H_0$ of order $(t-1)$.
\end{proof}


\begin{remark} \label{best_possible_estimate2}
Similarly to Remark \ref{best_possible_estimate1}, it is easy to see that the estimate in Corollary \ref{limits_slices_under_objects} is best possible in general (consider, for example, the case where $G \colon \Delta^0 \to \C$ is the inclusion of a weakly terminal object of order $0$ and $K = \varnothing$). For a different example, suppose that $\D$ is (the nerve of) an ordinary category with weak pushouts (of order $0$) and let $G \colon \D^{\op} \to \Set$ be a representable functor. Note that $G$ preserves weak pullbacks (of order $0$). Then the category $G_{\ast/}$, that is, the category of elements associated to $G$, admits only weak pullbacks of order $(-1)$ in general, that is, cones of this type exist but have no additional properties.  
\end{remark}

\section{Adjoint functor theorems} \label{AFTs}

\subsection{Adjoint functors} 
We recall the definition of adjoint functors between $\infty$-categories as well as some standard criteria for recognising adjoint functors (see \cite[5.2]{HTT}
and \cite[Chapter 6]{Ci} for a detailed exposition). 

\smallskip

For a map of simplicial sets $q \colon \mathscr{M} \to \Delta^1$ we write $\mathscr{M}_0$ (resp. $\mathscr{M}_1$) for the fiber over $0 \in \Delta^1$ (resp. the fiber over $1 \in \Delta^1$). The map $q$ is called \textit{bicartesian} if it is both cartesian and cocartesian. Note that in this case $q$ determines functors $\mathscr{M}_0 \to \mathscr{M}_1$ and $\mathscr{M}_1 \to \mathscr{M}_0$, essentially uniquely. 

\begin{definition}
Let $\C$ and $\D$ be $\infty$-categories. An \emph{adjunction} between $\C$ and $\D$ consists of a bicartesian fibration $q \colon \mathscr{M} \to \Delta^1$ and equivalences $\C \simeq \mathscr{M}_0$ and $\D \simeq \mathscr{M}_1$. These data determine functors, essentially uniquely,
\[
F \colon \C \to \D
\]
and
\[
G \colon \D \to \C.
\]
Then we say that $F$ is left adjoint to $G$ (resp. $G$ is right adjoint to $F$).
\end{definition}

A different characterization of adjoint functors between $\infty$-categories, which is analogous to the usual definition of an adjunction in ordinary category theory, is as follows: Given a pair of functors $F \colon \C \rightleftarrows \D \colon G$, $F$ is left adjoint to $G$ if and only if there is a natural transformation 
$$u \colon \mathrm{id}_{\C} \to G \circ F$$
such that the composition 
$$\mathrm{map}_{\D}(F(c), d) \stackrel{G}{\to} \mathrm{map}_{\C}(G(F(c)), G(d)) \stackrel{u^\ast}{\to} \mathrm{map}_{\C}(c, G(d))$$
is a weak equivalence for all $c \in \C$ and $d \in \D$. The natural transformation $u$ is the \emph{unit transformation} of the adjunction 
and it can be constructed using the bicartesian properties of the adjunction (see \cite[Proposition 5.2.2.8]{HTT}). If it exists, an adjoint of a functor is uniquely determined \cite[Proposition 5.2.6.2]{HTT}.

\bigskip

We will make use of the following useful characterizations for the recognition of adjoint functors.

\begin{proposition} \label{characterization_adjoint_functor}
Let $G \colon \D \to \C$ be a functor between $\infty$-categories. Then the following are equivalent:
\begin{enumerate}
\item The functor $G$ admits a left adjoint. 
\item The $\infty$-category $G_{c/}$ has an initial object for every object $c \in \C$.
\item The functor
\[
\mathrm{map}_{\C}(c,G(-)) \colon \D \to \S
\] 
is corepresentable for every object $c \in \C$.
\end{enumerate}
Dually, we have the following statement: \\ Let $F \colon \C \to \D$ be a functor between $\infty$-categories. Then the following are equivalent:
\begin{enumerate}[label={(\arabic*$'$})]
\item The functor $F$ admits a right adjoint.
\item The $\infty$-category $F_{/d}$ has a terminal object for every object $d \in \D$.
\item The functor
\[
\mathrm{map}_{\D}(F(-),d) \colon \C^{op} \to \S
\]
is representable for every object $d \in \D$.
\end{enumerate}
\end{proposition}
\begin{proof}
This characterization is obtained in \cite[Proposition 6.1.11]{Ci}. It is also a reformulation of \cite[Lemma 5.2.4.1]{HTT} using \cite[Proposition 4.4.4.5]{HTT} and \cite[Propositions 4.2.1.5 and 4.2.1.6]{HTT}. 
\end{proof}

Left adjoint functors preserve colimits and right adjoint functors preserve limits (see \cite[Proposition 5.2.3.5]{HTT}, \cite[Proposition 6.2.15]{Ci}). As the next proposition shows, this property of adjoint functors extends to higher weak (co)limits.

\begin{proposition} \label{adjoint_preserve_weak_co_limits}
Let $F \colon \C \rightleftarrows \D \colon G$ be an adjunction between $\infty$-categories. Let $K$ be a simplicial set and let $t \geq -1$ be an integer or $t = \infty$. Then the following hold:
\begin{enumerate}
\item The left adjoint functor $F$ preserves weak $K$-colimits of order $t$. 
\item The right adjoint functor $G$ preserves weak $K$-limits of order $t$. 
\end{enumerate}
\end{proposition}
\begin{proof}
We just prove (1), since the proof of (2) is completely analogous. Let $H \colon K^{\triangleright} \to \C$ be a weak colimit of order $t$ on a diagram $H_0:=H|_{K} \colon K \to \C$; let $x \in \C$ denote the cone object of $H$. Recall that the adjunction $F \colon \C \rightleftarrows \D \colon G$ induces an adjunction on functor $\infty$-categories:
\[
F^X \colon \C^X \rightleftarrows \D^X \colon G^X
\]
for any simplicial set $X$. Moreover, the adjunction equivalences on mapping spaces are natural in $X$ (see \cite[Theorem 6.1.22]{Ci}). This implies that for every object $y \in \D$ there is commutative diagram in the homotopy category of spaces as follows,
\[
\begin{tikzcd}
\mathrm{map}_{\D}(F(x),y) \arrow[d,dash,"\simeq"'] & \mathrm{map}_{\D^{K^{\triangleright}}}(F \circ H,c_y) \arrow[d,dash,"\simeq"'] \arrow[l,"\simeq"] \arrow[r] & \mathrm{map}_{\D^K}(F \circ H_0,c_y) \arrow[d,dash,"\simeq"] \\
\mathrm{map}_{\C}(x,G(y)) & \mathrm{map}_{\C^{K^{\triangleright}}}(H,c_{G(y)}) \arrow[l,"\simeq"] \arrow[r] & \mathrm{map}_{\C^K}(H_0,c_{G(y)}),
\end{tikzcd}
\]
where the vertical maps are the adjunction equivalences, the left horizontal maps are the restrictions along the inclusion of the cone point, and the right horizontal maps are the restrictions along the inclusion $K \subset K^{\triangleright}$; $c_y$ (resp. $c_{G(y)}$) denotes the constant diagram at $y \in \D$ (resp. $G(y) \in \C$). By Proposition \ref{characterization_weak_colimit}, the lower horizontal composition is $t$-connected, therefore the upper horizontal composition is also $t$-connected. Applying Proposition \ref{characterization_weak_colimit} again, we conclude that the cone $F \circ H \colon K^{\triangleright} \to \D$ is a weak colimit of order $t$ on the diagram $F \circ H_0 \colon K \to \D$. 
\end{proof} 

\subsection{Criteria for the existence of initial objects} In this subsection, we will establish some general criteria for the existence of initial objects in an $n$-category. These criteria are essentially refinements and direct generalizations of analogous criteria that appeared in our previous work \cite{NRS}, but there are also some interesting differences, especially, in the proofs of these generalizations. Based on the characterizations of adjoint functors in Proposition \ref{characterization_adjoint_functor}, the criteria of this subsection will be used later to obtain new adjoint functor theorems for $n$-categories. As usual, there are dual statements to the results of this section, concerning criteria for the existence of terminal objects as well as corresponding dual adjoint functor theorems; we will refrain from stating these explicitly in order to simplify the exposition. 

\medskip

We begin by defining the following properties of objects in an $\infty$-category which are weak versions of the property of being an initial object 
(cf. \cite[Section 2]{NRS}).

\begin{definition}\label{def_initial_objs_types}
Let $\C$ be an $\infty$-category. 
\begin{enumerate}
\item A collection of objects $S$ in $\C$ is called \emph{weakly initial} if for every object $y \in \C$, there is $x \in S$ such that $\map_{\C}(x, y)$ is non-empty.  An object $x \in \C$ is \emph{weakly initial} if the set $S = \{x\}$ is weakly initial. (In other words, $x$ is weakly initial in $\C$ of order $0$.)
\item An object $x \in \C$ is \emph{hypoinitial} if there is a cone $C_{x} \colon \C^{\triangleleft} \to \C$ on the identity  
$\mathrm{id} \colon \C \to \C$ with cone object $x$.
\item An object $x \in \C$ is \emph{h-initial} if $x$ is weakly initial of order $1$, that is, if the mapping space $\map_{\C}(x, y)$ is non-empty and connected for any $y \in \C$.
\end{enumerate}
\end{definition}

We collect below some easy observations about the comparison between these different notions:
\begin{itemize}
\item[(a)] Every initial object satisfies the properties in (1)--(3) above. (For (2), see also Proposition \ref{hypoinitial_vs_initial}.)
\item[(b)] Every hypoinitial or $h$-initial object is also weakly initial.
\item[(c)] An object in $\C$ is $h$-initial if and only if it is initial in $\h(\C)$. 
\end{itemize} 

\begin{example} \label{initial_example1}
Let $\C$ be an $\infty$-category with an initial object $0$. Then an object $x \in \C$ is hypoinitial if (and only if) there is a morphism $f \colon x \to 0$ in $\C$.  To see this, it suffices to consider the limit cone $\C^{\triangleleft} \to \C$ on the identity $\mathrm{id}\colon \C \to \C$ with cone object $0$ and form the essentially unique cone which is obtained by precomposition with the morphism $f$. Thus, a hypoinitial object need not be initial or $h$-initial in general.  
\end{example}

\begin{example} \label{initial_example2}
Let $\C$ be the (ordinary) category of non-empty sets. Then every object $x \in \C$ is weakly initial, but $\C$ does not contain hypoinitial or ($h$-)initial objects. To see this, note that if $C_x \colon \C^{\triangleleft} \to \C$ were a cone manifesting $x \in \C$ as a (hypo)initial object, then each of its components $C_x(y) \colon x \to y$ would have to factor through every inclusion $\{*\} \to y$, which is impossible. Thus, a weakly initial object does not ensure the existence of a hypoinitial object in general.   
\end{example}

\begin{example} \label{initial_example3}
Let $\C$ be the (ordinary) category of pointed sets which contain at least two elements. Every object $x \in \C$ is hypoinitial (and therefore also weakly initial) because $x$ is the cone object of the cone with zero maps as components. However, $\C$ does not have an initial object. 
\end{example}

The previous elementary examples demonstrate some of the basic differences between these different notions. Our goal in this subsection is to prove that a weakly initial (resp. hypoinitial, $h$-initial) object implies under appropriate assumptions the existence of an initial object. These will provide useful criteria for 
the existence of initial objects.

\medskip

The first criterion is concerned with the comparison between $h$-initial and initial objects (cf. \cite[Proposition 2.2.2]{NRS}).

\begin{lemma}[Criterion A]\label{criterionA}
Let $\C$ be an $n$-category, where $n \geq 1$ is an integer or $n=\infty$. Suppose that $\C$ admits weak $K$-limits of order $0$ for every finite simplicial set $K$ of dimension $\leq n-1$. Then an object $x \in \C$ is $h$-initial in $\C$ if and only if $x \in \C$ is initial in $\C$.
\end{lemma}
\begin{proof}
An initial object is obviously $h$-initial. For the converse, let $x \in \C$ be an $h$-initial object and let $y \in \C$ be an arbitrary object. We are required to prove that the mapping space $\map_{\C}(x,y)$ is contractible.  Since $\C$ is an $n$-category, it is enough to show that $\map_{\C}(x,y)^{S^k}$ is connected for $0 \leq k < n$. By assumption, the objects of $\C$ admit weak cotensors with $S^k$ (of order $0$) for all $0 \leq k <n$. (See \cite[4.4.4]{HTT} for the dual notion of tensoring with a space defined as an example of a colimit.) In other words, using Proposition \ref{characterization_weak_colimit}, this means that for each such $k \leq n-1$, the constant $S^k$-diagram at $y \in \C$ admits a cone with cone object $y^{S^k_w}$, such that the following canonical map is surjective:
\[
\pi_0(\map_{\C}(x,y^{S^k_w})) \to \pi_0(\map_{\C}(x,y)^{S^k}).
\]
Since $x \in \C$ is $h$-initial, the domain of this surjective map is a singleton. This implies that the target must also be a singleton, which means that the space $\map_{\C}(x, y)^{S^k}$ is connected.
\end{proof}

\begin{example} \label{example_CriterionA}
Suppose that $\C$ is an $n$-category which admits finite products and weak pullbacks of order $(n-2)$. Then, by Proposition \ref{weakly_cocomplete_colimits}, for any finite simplicial set $K$ of dimension $d$, the $n$-category $\C$ admits weak $K$-limits of order $(n-1-d)$.
In particular, $\C$ satisfies the assumptions of Lemma \ref{criterionA}.
\end{example}

The comparison between the notions of initial and hypoinitial objects is explained further in the following proposition. We denote by $\N^{\op}$ the opposite category associated with the poset of non-negative integers with its usual ordering. 

\begin{proposition} \label{hypoinitial_vs_initial}
Let $\C$ be an $\infty$-category.  
\begin{enumerate}
\item An object $x \in \C$ is initial if and only if there is a cone $C_{x} \colon \C^{\triangleleft} \to \C$ on the identity $\mathrm{id} \colon \C \to \C$ with cone object $x \in \C$, such that the component at $x$, that is, the associated morphism in $\C$,
$$C_{x}(x) \colon x \to x,$$
is an equivalence. Moreover, in this case, the cone $C_{x}$ is a limit of the identity $\mathrm{id} \colon \C \to \C$. 
\item Suppose that $x \in \C$ is hypoinitial and let $C_{x} \colon \C^{\triangleleft} \to \C$ be a cone on the identity $\mathrm{id} \colon \C \to \C$ with cone object $x$. Then the component of $C_{x}$ at $x \in \C$, that is, the associated morphism
$$C_{x}(x) \colon x \to x,$$
is an idempotent in $\C$. 
\item Suppose that $\C$ admits weak $\N^{\op}$-limits of order $1$. Then $\C$ has a hypoinitial object if and only if $\C$ has an initial object. 
\end{enumerate}
\end{proposition}
\begin{proof}
(1) The first claim follows from \cite[Proposition 4.2]{Jo} or \cite[Lemma 4.2.3]{RV} (see also \cite[Proposition 2.1.1]{NRS}). The second claim is shown in \cite[Proposition 2.1.1]{NRS}. 

\smallskip

(2): Let $\Idem$ denote the nerve of the category with one object $X$ and one (non-identity) idempotent morphism $e \colon X \to X$ as defined in \cite[4.4.5]{HTT}. We are required to construct a functor $\Idem \to \C$ which sends $X$ to $x \in \C$ and $e$ to the morphism $C_{x}(x)$. This functor will be constructed inductively on the skeletal filtration of $\Idem$. The simplicial set $\Idem$ contains exactly one non-degenerate simplex $\sigma_n$ in each simplicial degree $n \geq 0$; $\sigma_n$ corresponds to the string of $n$ copies of $e$. Note that each of the faces of $\sigma_n$, $n \geq 1$, is given by the $(n-1)$-simplex $\sigma_{n-1}$. 

First we define a diagram $F^{(1)} \colon \sk_1(\Idem) \to \C$ that sends $\sigma_0$ (i.e., the object $X$) to $x \in \C$ and $\sigma_1$ (i.e., the morphism $e$) to $C_{x}(x)$. For $n > 1$, assuming that the diagram $F^{(n-1)} \colon \sk_{n-1}(\Idem) \to \C$ has already been constructed, we define an extension 
$$F^{(n)} \colon \sk_n(\Idem) \to \C$$
by sending $\sigma_n$ to the following $n$-simplex of $\C$:  
$$F^{(n)}(\sigma_n) \colon \Delta^n \cong \Delta^0 \ast \Delta^{n-1} \xrightarrow{\mathrm{id} \ast F^{(n-1)}(\sigma_{n-1})} \Delta^0 \ast \C \xrightarrow{C_x} \C.$$ 
By the definition of this inductive process, we observe that the faces of $F^{(n)}(\sigma_n)$ are given by $F^{(n-1)}(\sigma_{n-1})$, therefore $F^{(n)}$ is well-defined. This completes the inductive construction of $F^{(n)} \colon \sk_n(\Idem) \to \C$ for all $n \geq 0$. Passing to the colimit of the skeletal filtration of $\Idem$, we obtain the required functor $F \colon \Idem \to \C$.  

\smallskip

(3): It follows from (1) that every initial object is hypoinitial. Conversely, let $x \in \C$ be a hypoinitial object and let $C_x \colon \C^{\triangleleft} \to \C$ be a cone on the identity $\id \colon \C \to \C$ with cone object $x$. We denote by $e \colon x \to x$ the component of $C_x$ 
at $x \in \C$. 

Let $\mathrm{Spine}_{\infty}$ denote the infinite spine
\[
\mathrm{Spine}_{\infty}:=\Delta^1 \amalg_{\Delta^0} \Delta^1 \amalg_{\Delta^0}...,
\]
and note that this is Joyal equivalent to $\N$. Using our assumptions on $\C$, the $\mathrm{Spine}_{\infty}^{\op}$-diagram in $\C$,
\[
E \colon \begin{tikzcd} \cdots \arrow[r,"e"] & x \arrow[r,"e"] & x \arrow[r,"e"] & x,
\end{tikzcd}
\]
admits a weak limit of order $1$, denoted by
\[
{}^y E \colon (\mathrm{Spine}_{\infty}^{\op})^{\triangleleft} \to \C,
\]
where $y \in \C$ is the cone object. We will show that $y \in \C$ is an initial object using the characterization in (1). The $0$-th/bottom component of the cone ${}^y E$ is a morphism
\[
i \colon y \to x
\]
(Since $e$ is idempotent by (2), it is easy to show that any other component of ${}^y E$ is canonically identified with $i$. In particular, $[i] = [e] \circ [i]$ in $\h(\C)$.) We may postcompose $i$ with the cone $C_x$ to obtain a new cone on the identity $\id \colon \C \to \C$, denoted by 
\[
C_{y} \colon \C^{\triangleleft} \to \C,
\]
whose cone object is $y \in \C$. In particular, $y$ is hypoinitial. Therefore, by (1), it suffices to show that the component $C_y(y) \colon y \to y$ of the cone $C_y$ at $y$ is an equivalence. Let $C_y(x) \colon y \to x$ denote the component of $C_y$ at $x$. It follows by direct inspection that $[C_y(x)] = [e] \circ [i]$, as a consequence of the definition of $C_y$. But $[e] \circ [i] = [i]$, so $[C_y(x)] = [i]$. 

Consider a new cone ${}^x E$ on $E$, defined by 
$${}^x E \colon \Delta^0 \ast \mathrm{Spine}_{\infty}^{\op} \xrightarrow{\id \ast E} \Delta^0 \ast \C \xrightarrow{C_x} \C$$
whose cone object is $x \in \C$. Moreover, the cone ${}^x E$ is a weak limit of $E$ of order $0$. To see this, we appeal to the characterization of Proposition 
\ref{characterization_weak_colimit} and claim that the canonical map (defined by the cone ${}^x E$) 
\begin{equation} \label{canonical_map_in_proof} \tag{*}
\map_{\C}(z, x) \rightarrow \mathrm{lim}_{\mathrm{Spine}_{\infty}^{\op}} \map_{\C}(z, x)
\end{equation}
is $0$-connected for any $z \in \C$. This holds because the canonical map \eqref{canonical_map_in_proof} can be identified up to homotopy with the retraction associated with the idempotent on the space $\map_{\C}(z, x)$ which is defined by composition with $e$ (see \cite[4.4.5.10--4.4.5.15]{HTT} for more details). 

There is a morphism (unique up to homotopy) of cones ${}^x E \to {}^y E$ defined by the following diagram:
$$\Delta^0 \ast (\Delta^0 \ast \mathrm{Spine}_{\infty}^{\op}) \xrightarrow{\id \ast {}^y E} \Delta^0 \ast \C \xrightarrow{C_x} \C.$$
Let $r \colon x \to y$ denote the associated morphism between cone objects. (By construction, this is just $C_x(y)$. In particular, we have $[e] = [i] \circ [r]$.) 

There is also a morphism of cones ${}^y E \to {}^x E$, since ${}^x E$ is a weak limit of order $0$. This morphism may not be unique, but we may choose this so that the associated morphism of cone objects is $i \colon y \to x$. To see this, note that precomposition of ${}^x E$ with $i$ defines a cone ${}^yE'$
with cone object $y$, as well as a morphism of cones ${}^y E' \to {}^x E$. In order to identify ${}^y E'$ and ${}^y E$, we observe that the space of cones on $E$ with cone object $y$ is given by the retract
$$\mathrm{lim}_{\mathrm{Spine}_{\infty}^{\op}} \map_{\C}(y, x) \subseteq \map_{\C}(y, x).$$
Thus, two such cones are equivalent if and only if their components are homotopic. This is obviously satisfied by the pair of cones ${}^y E'$ and ${}^y E$.

The composite morphism ${}^y E \to {}^x E \to {}^y E$ 
is (non-canonically) equivalent to the identity morphism, since ${}^y E$ is a weak limit of order $1$. In particular, $[r] \circ [i] = \id_y$.  

Combining these observations, we conclude the following identification of $C_y(y)$:
$$[C_y(y)] = [r] \circ [C_y(x)] = [r] \circ [i] = \id_y.$$
In particular, $C_y(y)$ is an equivalence, and therefore $y$ is initial in $\C$ by the characterization in (1).  
\end{proof}

\begin{example}
Example \ref{initial_example3} shows that the additional assumption in Proposition \ref{hypoinitial_vs_initial}(3) is required. 
\end{example}

The following statement provides conditions under which a weakly initial set implies the existence of a hypoinitial object. Moreover, combined with Proposition \ref{hypoinitial_vs_initial}, it gives our second criterion for the existence of initial objects. 

\begin{lemma}[Criterion B] \label{criterionB} 
Let $\C$ be a locally small $n$-category, where $n \geq 1$ is an integer or $n=\infty$. Suppose that $\C$ has small products and weak pullbacks of order $(n-2)$. Then the following statements hold:
\begin{enumerate}
\item $\C$ has a small weakly initial set if and only if $\C$ has a hypoinitial object. 
\item Suppose that $\C$ has weak $\N^{\op}$-limits of order $1$. Then $\C$ has a small weakly initial set if and only if $\C$ has an initial object. 
\end{enumerate}
\end{lemma}

For the proof of Lemma \ref{criterionB}, we will need the following technical fact which is a generalization of \cite[Lemma 5.4.5.10]{HTT}. First we recall some terminology. We say that a functor $F \colon \C \to \D$ between $\infty$-categories 
is a (categorical) $n$-\emph{equivalence} if the induced functor $\h_n(F) \colon \h_n \C \to \h_n \D$ is an equivalence of $n$-categories. More generally, 
we say that a map $f \colon X \to Y$ of simplicial sets is a \emph{categorical n-equivalence} if it induces an $n$-equivalence between the associated $\infty$-categories, 
that is, after fibrant replacement in the Joyal model category. The class of $n$-equivalences of simplicial sets defines a left Bousfield localization of the 
Joyal model category whose fibrant objects are the $\infty$-categories which are equivalent to an $n$-category (see \cite{CL}).

\begin{lemma} \label{technical_lemma}
Let $X$ be an $(n-1)$-connected simplicial set, $n \geq 1$. Then the canonical inclusion map 
$$j \colon \Delta^0 \ast X \cup_{X} X \ast \Delta^0 \to \Delta^0 \ast X \ast \Delta^0$$
is a categorical $n$-equivalence. 
\end{lemma}

\smallskip

For the proof of Lemma \ref{technical_lemma}, we will make use of the model for mapping spaces which was introduced and developed in the work of Dugger--Spivak \cite{DS1, DS2}. We recall that a \emph{necklace} is a simplicial set of the form 
$$T = \Delta^{n_0} \vee \Delta^{n_1} \vee \cdots \vee \Delta^{n_k}$$
where $n_i \geq 0$ and where the final vertex of $\Delta^{n_i}$ is identified with the initial vertex of $\Delta^{n_{i+1}}$ for each $0 \leq i \leq k-1$.  Every necklace $T$ is equipped with a map $\partial \Delta^1 \to T$ which sends $0$ to the initial vertex $\alpha$ of the necklace $T$ and $1$ to the final vertex $\omega \in T$. The category 
$\Nec$ of necklaces is the full subcategory of $\SSet_{\ast, \ast} = (\partial \Delta^1 \downarrow \SSet)$ which is spanned by the necklaces. In other words, 
a map of necklaces is a map between the underlying simplicial sets that preserves both basepoints. 

Let $\SCat$ denote the category of simplicial categories (or simplicially enriched categories) and let  $\CC \colon \SSet \to \SCat$ denote the left adjoint of 
the coherent nerve functor $N_{\Delta}$ (see \cite[1.1.5]{HTT}, \cite{DS1}). By \cite[Theorem 5.3]{DS1}, given a simplicial set $X$ and $x, y \in X$, there is a canonical 
(zigzag of) weak equivalence(s) of simplicial sets:
\begin{equation} \label{DS_mapspace} \tag{*}
\Map_{\CC(X)}(x, y) \simeq N(\Nec \downarrow X_{x,y})
\end{equation}
where $(X_{x,y} \colon \partial \Delta^1 \to X) \in \SSet_{\ast, \ast}$ denotes the simplicial set $X$ with $x,y \in X$ as basepoints -- sending $0 \in \partial \Delta^1$ to $x$ and $1 \in \partial \Delta^1$ to $y$ -- and $N(\Nec \downarrow X_{x, y})$ is the usual nerve of the ordinary slice category of necklaces over the object $X_{x,y} \in \SSet_{\ast, \ast}$ (this is the same as $J_{/X_{x,y}}$ where $J \colon \Nec \to \SSet_{\ast, \ast}$ denotes the inclusion functor). 

Moreover, as a consequence of fundamental well-known properties of the adjunction $(\CC, N_{\Delta})$, the mapping space $\Map_{\CC(X)}(x, y)$ is canonically weakly equivalent to the mapping space from $x$ to $y$ in the $\infty$-category associated to $X$ (see \cite[2.2]{HTT}, \cite{DS2}).

\medskip

\noindent \emph{Proof of Lemma \ref{technical_lemma}}.
Consider the induced functor between the associated simplicial categories
$$\CC(j) \colon \MM = \CC(\Delta^0 \ast X \cup_X X \ast \Delta^0) \to \NN = \CC(\Delta^0 \ast X \ast \Delta^0)$$
where $\CC$ is the rigidification functor which is left adjoint to the coherent nerve functor $N_{\Delta}$. Based on well-known properties of the adjunction $(\CC, N_{\Delta})$ \cite{HTT, DS2}, the claim in the lemma is equivalent to the claim that the functor $\CC(j)$ is essentially surjective and induces $\pi_*$-equivalences on all mapping spaces for $* \leq (n-1)$ (and all basepoints). 

The functor $\CC(j)$ is bijective on objects. The set of objects consists of the left cone object $0$, the $0$-simplices $x \in X$, and the right cone object $1$. The mapping spaces of $\NN$ are easy to identify:
\begin{itemize}
\item[(i)] $\Map_{\NN}(0, x)$ is weakly contractible for any $x \in X$. 
\item[(ii)] $\Map_{\NN}(x, 1)$ is weakly contractible for any $x \in X$. 
\item[(iii)] For $x, y \in X$, $\Map_{\NN}(x, y)$ is equivalent to the mapping space in the $\infty$-category associated with $X$.
\item[(iv)] $\Map_{\NN}(0,1)$ is weakly contractible. 
\end{itemize}
 To see (i)--(iv), it suffices to identify these mapping spaces with the mapping spaces in the $\infty$-category $\Delta^0 \ast \C \ast \Delta^0$ where $\C$ denotes a fibrant replacement of $X$ in the Joyal model category; this also uses well-known properties of the adjunction $(\CC, N_{\Delta})$ (see \cite[2.2]{HTT}, \cite{DS2}). 

On the other hand, the mapping spaces of the simplicial category 
\begin{equation*} \label{pushout_M} 
\MM = \CC(\Delta^0 \ast X \cup_X X \ast \Delta^0) \cong \CC(\Delta^0 \ast X) \cup_{\CC(X)} \CC(X \ast \Delta^0)
\end{equation*}
are identified as follows:
\begin{itemize}
\item[(i)$'$] $\Map_{\MM}(0, x)$ is weakly contractible for any $x \in X$. 
\item[(ii)$'$] $\Map_{\MM}(x, 1)$ is weakly contractible for any $x \in X$. 
\item[(iii)$'$] For $x, y \in X$, $\Map_{\MM}(x, y)$ is equivalent to the mapping space in the $\infty$-category associated with $X$.
\end{itemize}
To see (i)$'$--(iii)$'$, it suffices to observe that these mapping spaces in the pushout $\MM$ agree with the corresponding mapping spaces in $\CC(\Delta^0 \ast X)$, $\CC(X)$, or $\CC(X \ast \Delta^0)$, and the latter mapping spaces can be identified similarly to (i)--(iv) above. Moreover, it is easy to see from the identifications in (iii) and (iii)$'$ that $\CC(j)$ induces weak equivalences on mapping spaces $\Map_{\MM}(x, y) \simeq \Map_{\NN}(x,y)$ for all $x, y \in X$.  

The situation is different for the mapping space $\Map_{\MM}(0,1)$, which is not weakly contractible in general. In order to identify this mapping space, we will make use of the model via necklaces \eqref{DS_mapspace} from \cite{DS1}. Let us write $\widetilde{X}_{0,1}$ for $\Delta^0 \ast X \cup_X X \ast \Delta^0$ equipped with the two basepoints given by the cone objects. According to \cite[Theorem 5.3]{DS1}, the mapping space $\Map_{\MM}(0,1)$ is weakly equivalent to 
$$N(\Nec \downarrow \widetilde{X}_{0,1}).$$
Let $\Nec'$ denote the full subcategory of $\SSet$ which is spanned by necklaces (after forgetting the basepoints). 
There is a functor $\psi \colon (\Nec \downarrow \widetilde{X}_{0,1}) \to (\Nec' \downarrow X)$ which sends $(T, b \colon T \to \widetilde{X}_{0,1})$ to $(T_X = b^{-1}(X), \ b_{|T_X} \colon T_X \to X)$ -- note that $T_X$ is a full connected simplicial subset of $T$ and therefore again a necklace. We claim that the induced map
$$N(\psi) \colon  N(\Nec \downarrow \widetilde{X}_{0,1}) \to N(\Nec' \downarrow X)$$ 
is a weak equivalence of simplicial sets. By Quillen's Theorem A (see, for example, \cite[Corollary 4.1.3.3]{HTT}), it suffices to prove that for every $(S, q \colon S \to X) \in (\Nec' \downarrow X)$, the slice category $\psi_{/(S, q)}$ has weakly contractible nerve. This slice category can be identified more explicitly as follows. Let us write $\widetilde{S}_{0,1} = \Delta^0 \ast S \cup_S S \ast \Delta^0$ for brevity (with basepoints given by the two cone objects). Then we consider the following functors:

\medskip

\noindent \emph{The functor} $\nu \colon \psi_{/(S, q)} \longrightarrow (\Nec \downarrow \widetilde{S}_{0,1})$. An object in $\psi_{/(S, q)}$ is given by an object $(T , b \colon T \to \widetilde{X}_{0,1})$ in $(\Nec \downarrow \widetilde{X}_{0,1})$ together with a morphism in $(\Nec' \downarrow X)$:
$$
\xymatrix{ 
T_X = b^{-1}(X) \ar[rr]^(0.6){u} \ar[dr]_{b_{|T_X}} && S \ar[dl]^q \\
& X. & \\
}
$$
The functor $\nu$ sends an object $(T, b; u)$ to the object $(T, c \colon T \to \widetilde{S}_{0,1})$ in $(\Nec \downarrow \widetilde{S}_{0,1})$ defined by
$$c \colon T \to \Delta^0 \ast T_X \cup_{T_X} T_X \ast \Delta^0 \xrightarrow{\id \ast u \cup u \ast \id} \Delta^0 \ast S \cup_S S \ast \Delta^0 = \widetilde{S}_{0,1}$$
where the first map sends $b^{-1}(0)$ to $0$, $b^{-1}(1)$ to $1$, and it is the canonical inclusion on $T_X $. (It is helpful to recall here that the join functor $\ast \colon \SSet_{/\partial \Delta^1} \to \SSet_{/\Delta^1}$ is the right adjoint of the pullback functor $i^* \colon \SSet_{/\Delta^1} \to \SSet_{/\partial \Delta^1}$, where $i \colon \partial \Delta^1 \to \Delta^1$ is the boundary inclusion.)

\medskip

\noindent \emph{The functor} $\mu \colon (\Nec \downarrow \widetilde{S}_{0,1}) \longrightarrow \psi_{/(S,q)}$. This sends an object $(T, c \colon T \to \widetilde{S}_{0,1})$ to the object $(T, b; u)$ which consists of
$$(T, \ b \colon T \xrightarrow{c} \widetilde{S}_{0,1} \xrightarrow{\id \ast q \cup q \ast \id} \widetilde{X}_{0,1}) \in (\Nec \downarrow \widetilde{X}_{0,1})$$
and the morphism in $(\Nec' \downarrow X)$: 
$$
\xymatrix{
T_X = c^{-1}(S) \ar[dr]_{b_{|T_X}} \ar[rr]^(0.6){c_{|c^{-1}(S)}} && S \ar[dl]^q \\
& X. & \\
}
$$
It is easy to see that the functors $\nu$ and $\mu$ are inverse equivalences of categories. As a consequence, $N(\psi_{/(S,q)}) \simeq N(\Nec \downarrow \widetilde{S}_{0,1})$. The simplicial set $N(\Nec \downarrow \widetilde{S}_{0,1})$ models the mapping space from $0$ to $1$ in (the $\infty$-category associated to) $\Delta^0 \ast S \cup_S S \ast \Delta^0$. Using \cite[Lemma 5.4.5.10]{HTT}, this mapping space is weakly contractible since the necklace $S$ is weakly contractible. It follows that $N(\psi_{/(S,q)})$ is weakly contractible for every $(S, q) \in (\Nec' \downarrow X)$, and therefore, $N(\psi) \colon  N(\Nec \downarrow \widetilde{X}_{0,1}) \to N(\Nec' \downarrow X)$ is a weak equivalence, as claimed. 

Lastly, it remains to identify $N(\Nec' \downarrow X)$. There is an inclusion functor $$\tau \colon (\Delta \downarrow X) \to (\Nec' \downarrow X)$$ of those necklaces in $X$ of the form $\Delta^k \to X$. We claim that $\tau$ induces a weak equivalence of simplicial sets after passing to the nerves. By Quillen's Theorem A (see, for example, \cite[Corollary 4.1.3.3]{HTT}), it suffices to prove that the slice category $\tau_{/(T, b)}$ has weakly contractible nerve for every object $(T, b \colon T \to X)$ in $(\Nec' \downarrow X)$. Note that this slice category can be identified with the slice category $(\Delta \downarrow T)$. Since $N(\Delta \downarrow T)$ is weakly equivalent to $T$ and $T$ is weakly contractible, we can then conclude that $N(\tau)$ is a weak equivalence. Also, the simplicial set $N(\Delta \downarrow X)$ is weakly equivalent to $X$. 

In conclusion, the mapping space $\Map_{\MM}(0,1)$ is weakly equivalent to $X$. Since $X$ is $(n-1)$-connected by assumption,  this completes the proof of the lemma. 
\qed

\medskip

We are now ready to prove Lemma \ref{criterionB} (Criterion B) on the comparison between weakly initial sets and (hypo)initial objects. 

\medskip

\noindent \textbf{Proof of Lemma \ref{criterionB} (Criterion B).} We first prove (1). As remarked earlier, a hypoinitial object is clearly weakly initial, so one direction is obvious. Conversely, suppose that $\C$ has a small weakly initial set $S$. We denote by $\C_S$ the full subcategory of $\C$ which is spanned by the objects in $S$. Note 
that $\C_S$ is essentially small since $\C$ is locally small. We will prove the following two claims: 
\begin{enumerate}
\item[(i)] The inclusion $\iota \colon \C_S \to \C$ admits a cone $\iota^{\triangleleft} \colon \C_S^{\triangleleft} \to \C$. 
\item[(ii)] The inclusion $\iota \colon \C_S \to \C$ is coinitial (see \cite[4.1.1]{HTT}, \cite[2.4]{NRS} for the definition and properties of coinitial maps).
\end{enumerate}
Assuming (ii), it follows that the canonical ``restriction'' functor $\C_{/\id} \to \C_{/\iota}$ is an equivalence of $\infty$-categories. Assuming (i), there is a cone $\iota^{\triangleleft}$ which defines an object in $\C_{/\iota}$. Then this can be lifted to an object in $\C_{/\id}$, that is, a cone on the identity $\id \colon \C \to \C$, as required. Thus, the proof of (1) will be complete once we prove (i) and (ii). 

\medskip

\noindent \emph{Proof of (i).} Since $\C$ admits small products and weak pullbacks of order $(n-2)$, it follows by Proposition \ref{weakly_cocomplete_colimits} that every diagram $K \to \C$, where $K$ is a small simplicial set of dimension $\leq n$, admits 
a cone. Without loss of generality, we may assume that $\C_S$ is small. Then there exists a cone $F \colon \Delta^0 \ast \sk_n(\C_S) \to \C$ on the restriction of $\iota$ to the $n$-skeleton of $\C_S$,
\[
\mathrm{sk}_n(\C_S) \subset \C_S \xrightarrow{\iota} \C.
\]
In order to extend this cone to a cone on $\iota$, we need to find a diagonal filler in the following diagram:
\[
\begin{tikzcd}
\Delta^0 \ast \mathrm{sk}_n(\C_S) \cup_{\mathrm{sk}_n(\C_S)} \C_S \arrow[d,hookrightarrow] \arrow[r] & \C. \\
\Delta^0 \ast \C_S \arrow[ur,dashed]
\end{tikzcd}
\]
This extension problem is (uniquely) solvable, since $\C$ is an $n$-category and $\Delta^0 \ast \C_S$ is obtained from $\Delta^0 \ast \mathrm{sk}_n(\C_S) \cup_{\mathrm{sk}_n(\C_S)} \C_S$ by attaching simplices of dimension  $\geq n+2$. This proves (i).

\medskip

\noindent \emph{Proof of (ii).} By \cite[Theorem 4.1.3.1]{HTT}, the claim is equivalent to the claim that the slice $\infty$-category $\iota_{/x}$ is weakly contractible for every $x \in \C$. For this, it suffices to prove the following \emph{assertion}: every map $K \to \iota_{/x}$, where $K$ is a finite simplicial set, factors through a weakly contractible simplicial set.  We will first prove this assertion in two special cases: (a) when the dimension of $K$ is less than $n$, and (b) when $K$ is $(n-1)$-connected; then we will deduce the general case for a finite simplicial set $K$. 

\emph{Case (a)}: Suppose that the dimension of $K$ is $< n$.  Our assumptions on $\C$ imply that every $L$-diagram $L \to \C$, where $L$ is a small simplicial set of dimension $\leq n$, admits a cone (see Proposition \ref{weakly_cocomplete_colimits}). We claim that for every $K$-diagram $f \colon K \to \iota_{/x}$, there is an extension $\Delta^0 \ast K \to \iota_{/x}$. Let $f \colon K \to \iota_{/x}$ be such a diagram and consider the corresponding $K$-diagram in $\C_{/x}$,
$$K \xrightarrow{f} \iota_{/x} = \C_S \times_{\C} \C_{/x} \xrightarrow{\iota'}  \C_{/x}.$$
The adjoint of $\iota' f$ is a (right) cone $K \ast \Delta^0 \to \C$ in $\C$ with cone object $x \in \C$. Since the dimension of $K \ast \Delta^0$ is $\leq n$, the last conic diagram extends to a (left) cone $F \colon \Delta^0 \ast K \ast \Delta^0 \to \C$. Then the adjoint of $F$, also denoted here by
$$F \colon \Delta^0 \ast K \to \C_{/x},$$
is a cone on $\iota' f$ with some (underlying) cone object $y \in \C$. Since $S$ is weakly initial in $\C$, there is $s \in S$ and a morphism $u \colon s \to y$ in $\C$. Precomposition of $F$ with the morphism $u$ yields a new (left) cone on $\iota'f$,
$$F_s \colon \Delta^0 \ast K \to \C_{/x},$$
with cone object $s \in S$. Since $\C_S$ is a full subcategory of $\C$, the diagram 
$$\Delta^0 \ast K \xrightarrow{F_s} \C_{/x} \xrightarrow{q} \C$$
factors through the inclusion $\C_S \subset \C$. Thus, the cone $F_s$ lifts to a cone on $f$, 
$$\Delta^0 \ast K \to \iota_{/x},$$
which proves the claim. The simplicial set $\Delta^0 \ast K$ is weakly contractible, so this concludes the proof in Case (a). 

\emph{Case (b)}: Suppose that $K$ is $(n-1)$-connected. It follows from our assumptions on $\C$ that any $L$-diagram $L \to \C$, where $L$ is a small simplicial set of dimension $\leq n$, admits a cone (see Proposition \ref{weakly_cocomplete_colimits}). In fact, 
this conclusion holds for \emph{any} small simplicial set $L$ and $L$-diagram $f \colon L \to \C$. To see this, note that the composite diagram $\sk_n(L) \subset L \xrightarrow{f} \C$ admits a cone by Proposition \ref{weakly_cocomplete_colimits}, and this cone can then be extended to a cone on $f$ using that $\C$ is an $n$-category and the same argument as in the proof of (i) above. So for any map $f \colon K \to \iota_{/x}$, the composite map 
\begin{equation} \label{composite_dia} \tag{*}
K \xrightarrow{f} \iota_{/x} \xrightarrow{q'} \C_S \xrightarrow{\iota} \C
\end{equation}
admits a cone $F \colon \Delta^0 \ast K \to \C$. Here $q' \colon \iota_{/x} = \C_S \times_{\C} \C_{/x} \to \C_S$ denotes the right fibration which is the pullback of the canonical right fibration $q \colon \C_{/x} \to \C$. Let $y \in \C$ denote the cone object of $F$. Since $S$ is weakly initial, there is $s \in S$ and a morphism $u \colon s \to y$ in $\C$. This morphism can be extended essentially uniquely to a morphism of cones $F_s \to F$ on the diagram \eqref{composite_dia}, where $s$ is the cone object of $F_s$. Moreover, since $\C_S$ is a full subcategory of $\C$, the cone $F_s$ factors through the inclusion $\iota \colon \C_S \to \C$.  Thus, we obtain a (left) cone (also denoted) $F_s \colon \Delta^0 \ast K \to \C_S$ on the diagram $q'f$ with cone object $s \in \C_S$. On the other hand, the adjoint of the diagram $f$ gives a (right) cone $F_x \colon K \ast \Delta^0 \to \C$
on the diagram \eqref{composite_dia} with cone object $x \in \C$. The cones $F_s$ and $F_x$ determine the following diagram in $\C$,
\begin{equation} \label{composite_dia2} \tag{**}
\Delta^0 \ast K \cup_{K} K \ast \Delta^0 \xrightarrow{F_s \cup F_x} \C.
\end{equation}
The inclusion map 
$$j \colon \Delta^0 \ast K \cup_{K} K \ast \Delta^0 \to \Delta^0 \ast K \ast \Delta^0$$
is a categorical $n$-equivalence by Lemma \ref{technical_lemma}. Therefore, since $\C$ is an $n$-category, it follows that \eqref{composite_dia2} extends along the inclusion $j$. Hence we obtain a diagram 
$$\Delta^0 \ast K \ast \Delta^0 \to \C$$
whose left cone is $F_s$ and the right cone object is $x \in \C$. Thus, the adjoint diagram defines a cone on $f$, 
$$\Delta^0 \ast K \to \iota_{/x}.$$ 
Since the simplicial set $\Delta^0 \ast K$ is weakly contractible, this completes the proof in Case (b). 

Suppose now that $K$ is an arbitrary finite simplicial set and let $K \to \iota_{/x}$ be a $K$-diagram. Using the result of \emph{Case (a}) above, the composite diagram 
$$\sk_{n-1}(K) \subset K \to \iota_{/x}$$
factors through a weakly contractible simplicial set $C$. As a consequence, $K \to \iota_{/x}$ extends also to a diagram on the pushout  $L := K \cup_{\sk_{n-1}(K)} C$; in other words, we have a diagram as follows,
$$
\xymatrix{
\sk_{n-1}(K) \ar[r] \ar[d] & K \ar[r] \ar[d] & \iota_{/x} \\
C \ar[r] & L := K \cup_{\sk_{n-1}(K)} C \ar@{-->}[ru] &
}
$$
where $L$ is $(n-1)$-connected (since $\sk_{n-1}(K) \to K$ is $(n-1)$-connected). Using the result of \emph{Case (b)}, the induced diagram $L \to \iota_{/x}$ factors further through a weakly contractible simplicial set. This factorization produces the required factorization of $K \to \iota_{/x}$ through a weakly contractible simplicial set. 

This proves (ii) and also completes the proof of (1).  

\medskip

(2) is a consequence of (1) and Proposition \ref{hypoinitial_vs_initial}. \qed

\begin{corollary} \label{weakly_initial_vs_initial}
Let $\C$ be a locally small $n$-category, where $n \geq 3$ is an integer or $n=\infty$. Suppose that $\C$ has small products and weak pullbacks of order $(n-2)$. Then $\C$ has a small weakly initial set if and only if $\C$ has an initial object. 
\end{corollary}
\begin{proof}
Note that $\N^{\op}$ is Joyal equivalent to its spine, that is, the union of the $1$-simplices $(k + 1 \to k)$ for $k \geq 0$, which is a $1$-dimensional simplicial set. By Proposition \ref{weakly_cocomplete_colimits}, for any $1$-dimensional simplicial set $K$, $\C$ admits weak $K$-limits of order $(n-2) \geq 1$. So $\C$ admits weak $\N^{\op}$-limits of order $1$. Then the result follows directly from Lemma \ref{criterionB}(2).  
\end{proof}

\begin{example} \label{counterexample_lower_1}
Corollary \ref{weakly_initial_vs_initial} fails for $n =1$. For a counterexample, let $X$ be a set with two elements, consider the (ordinary) category $\Set_{X/}$ of sets under $X$, and let $\C$ be the full subcategory of $\Set_{X/}$ spanned by the non-initial objects. Then 
$\C$ inherits small products from $\Set_{X/}$. This uses the observation that a product of a (small) family of objects $(X \to X_i)_{i \in I}$ in $\Set_{X/}$ is initial if and only if exactly one of the objects is initial and every other object is terminal. Moreover, $\C$ admits a hypoinitial object: simply consider any object $(X \to Y)$ in $\Set_{X/}$ with a retraction $(Y \to X)$ (under $X$) and use the retraction map to define the required cone $\C^{\triangleleft} \to \C$ (cf. Example \ref{initial_example1}). This cone can be used to define weak pullbacks (or general weak limits) of order $(-1)$ for any diagram in $\C$. Thus, $\C$ satisfies the assumptions of Corollary \ref{weakly_initial_vs_initial} and has a weakly initial object, but $\C$ does not have an initial object.    
\end{example}

\begin{example} \label{counterexample_lower_2}
Corollary \ref{weakly_initial_vs_initial} fails also for $n=2$. For a counterexample, let $X$ be a set with two elements, regarded as an object in the $\infty$-category $\S$ of spaces, and consider the homotopy $2$-category $\h_2(\S_{X/})$ of spaces under $X$. Let $\C$ denote the full subcategory of $\h_2(\S_{X/})$ spanned by the non-initial objects.  Then $\C$ inherits small products from $\S_{X/}$. Similarly to Example \ref{counterexample_lower_1}, this uses the observation that a product of a (small) family of objects $(X \to X_i)_{i \in I}$ in $\S_{X/}$ (or $\h_2(\S_{X/})$) is initial if and only if exactly one of the objects is initial and every other object is terminal. To see the existence of weak pullbacks of order $0$, note first that $\h_2(\S_{X/})$ admits weak pullbacks of order $1$ (Example \ref{weakly_cocomplete_example}) and a square in $\C$ is a weak pullback of order $1$ if (and only if) it is so in $\h_2(\S_{X/})$. Hence it remains to construct weak pullbacks of order $0$ for diagrams in $\C$
$$
\xymatrix{
& (B, X \to B) \ar[d] \\
(A, X \to A) \ar[r] & (C, X \to C) 
}
$$
which have the initial object $(X, X \xrightarrow{\mathrm{id}} X)$ as weak pullback of order $1$ in $\h_2(\S_{X/})$. Consider any object $(Y, i \colon X \to Y)$ in $\C$ with a retraction $r \colon Y \to X$; this defines a morphism to the initial object $(X, \mathrm{id}_X)$ in $\h_2(\S_{X/})$. Then precomposition with $r$ gives a square in $\C$:
 $$
 \xymatrix{
(Y, i \colon X \to Y)  \ar[r] \ar[d] & (B, X \to B) \ar[d] \\
 (A, X \to A) \ar[r] & (C, X \to C).
 }
 $$
The fact that $(X, \mathrm{id}_X)$ is a retract of $(Y, i)$ in $\h_2(\S_{X/})$ readily implies that the last square is a weak pullback of order $0$. Thus, $\C$ satisfies the assumption of Corollary \ref{weakly_initial_vs_initial} and the object $(Y, i \colon X \to Y)$ is weakly initial in $\C$ (of order $0$), by using the retraction $r$, but $\C$ does not have an initial object. 
\end{example}

\subsection{Adjoint functor theorems for $n$-categories} The characterization of adjoint functors in Proposition \ref{characterization_adjoint_functor} translates the problem of the existence of a right (left) adjoint to an associated problem about the existence of a terminal (initial) object. Using our two criteria for the existence of initial objects (Lemmas \ref{criterionA} and \ref{criterionB}), we can now formulate two general adjoint functor theorems
for functors between $n$-categories. These theorems generalize and refine the general adjoint functor theorems for $\infty$-categories shown in \cite{NRS} and provide useful characterizations of adjoint functors $G \colon \D \to \C$ between $n$-categories when $\D$ is weakly $n$-complete. The characterizations 
are given in terms of the following properties which were also considered for the adjoint functor theorems in \cite{NRS}. 

\begin{definition} \label{solution_set_conditions}
Let $\C$ and $\D$ be $\infty$-categories and let $G \colon \D \to \C$ be a functor. 
\begin{enumerate}
\item We say that $G$ satisfies the $h$-\emph{initial object condition} if the slice $\infty$-category $G_{c/}$ admits an $h$-initial object for every $c \in \C$. 
\item We say that $G$ satisfies the \emph{solution set condition} if the slice $\infty$-category $G_{c/}$ admits a small weakly initial set for every $c \in \C$. 
\end{enumerate}
\end{definition}

The first adjoint functor theorem builds on Criterion A (Lemma \ref{criterionA}) and generalizes \cite[Theorem 3.2.6]{NRS}. 

\begin{theorem}[\emph{n}-GAFT$_{\mathrm{fin}}$]\label{nGAFTfin}
Let $G \colon \D \to \C$ be a functor between $n$-categories, where $n \geq 1$ is an integer or $n= \infty$. Suppose that $\D$ is a finitely weakly complete $n$-category. Then $G$ admits a left adjoint if and only if $G$ preserves finite products, weak pullbacks of order $(n-1)$, and satisfies the $h$-initial object condition.
\end{theorem}
\begin{proof}
By Proposition \ref{characterization_adjoint_functor}, $G$ is a right adjoint if and only if $G_{c/}$ has an initial object for every $c \in \C$. 

If $G$ is a right adjoint, then $G$ clearly satisfies the $h$-initial object condition. Moreover, $G$ preserves finite products and weak pullbacks of order $(n-1)$ by Proposition \ref{adjoint_preserve_weak_co_limits}. 

Conversely, suppose that $G$ preserves finite products and weak pullbacks of order $(n-1)$, and $G_{c/}$ admits an $h$-initial object for every $c \in \C$. The case $n=1$ is obvious, so we may assume that $n \geq 2$. By Corollary \ref{limits_slices_under_objects}, the slice $n$-category $G_{c/}$ admits finite products and weak pullbacks of order $(n-2)$. Thus, using Proposition \ref{weakly_cocomplete_colimits}, it follows that $G_{c/}$ admits weak $K$-limits of order $0$ for any finite simplicial set $K$ of dimension $\leq n-1$ (see Example \ref{example_CriterionA}). Then the result follows from Lemma \ref{criterionA} (Criterion A).    
\end{proof}

The second adjoint functor theorem builds on Lemma \ref{criterionB} (Criterion B) and generalizes \cite[Theorem 3.2.5]{NRS}. We recall from \cite{NRS} that an $\infty$-category $\C$ is \emph{2-locally small} if for every pair of objects $x, y \in \C$, the mapping space $\map_{\C}(x, y)$ is locally small \cite[Definition 3.2.4]{NRS}.

\begin{theorem}[\emph{n}-GAFT] \label{nGAFT}
Let $G \colon \D \to \C$ be a functor between $n$-categories, where $n \geq 2$ is an integer or $n=\infty$. Suppose that $\D$ is a locally small weakly complete $n$-category and $\C$ is 2-locally small. Then $G$ admits a left adjoint if and only if $G$ preserves small products, weak pullbacks of order $(n-1)$, and satisfies the solution set condition.
\end{theorem}
\begin{proof}
By Proposition \ref{characterization_adjoint_functor}, $G$ is a right adjoint if and only if $G_{c/}$ admits an initial object for every $c \in \C$. Therefore, if $G$ is a right adjoint, then $G$ clearly satisfies the solution set condition. Moreover, $G$ preserves small products and weak pullbacks of order $(n-1)$ by Proposition \ref{adjoint_preserve_weak_co_limits}.

Conversely, suppose that $G$ preserves small products and weak pullbacks of order $(n-1)$, and the slice $n$-category $G_{c/}$ admits a small weakly initial set for every $c \in \C$. By Corollary \ref{limits_slices_under_objects}, the slice $n$-category $G_{c/}$ admits small products and weak pullbacks of order $(n-2)$. Moreover, $G_{c/}$ is locally small since $\D$ is locally small and $\C$ is 2-locally small (see \cite[Lemma 3.2.8]{NRS}). Then, for $n \geq 3$, Corollary \ref{weakly_initial_vs_initial} implies that $G_{c/}$ has an initial object for every $c \in \C$, and therefore $G$ admits a left adjoint. 

As we pointed out in Example \ref{counterexample_lower_2}, Corollary \ref{weakly_initial_vs_initial} fails in general for $n=2$, so we will need a separate argument for $n=2$. Even though $G_{c/}$ only admits weak $\N^{\op}$-limits of order $0$ in this case (using Proposition \ref{weakly_cocomplete_colimits} or Corollary \ref{limits_slices_under_objects}), the proof will be similar to the proof of Proposition \ref{hypoinitial_vs_initial}(3). (Essentially, this proof still works in this case because the weak $\N^{\op}$-limits of order $0$ in $G_{c/}$ can be constructed by lifting weak $\N^{\op}$-limits of order $1$ in $\D$ and the canonical left fibration $q \colon G_{c/} \to \D$ detects equivalences.) 

By Lemma \ref{criterionB}(1), the $\infty$-category $G_{c/}$ admits a hypoinitial object for every $c \in \C$. So there is cone $C_u \colon G_{c/}^{\triangleleft} \to G_{c/}$ on the identity functor with cone object denoted by $(x, u \colon c \to G(x)) \in G_{c/}$. By Proposition \ref{hypoinitial_vs_initial}(2), the 
component $C_u(x, u)$ of the cone $C_u$ at $(x, u \colon c \to G(x))$ is an idempotent morphism:
$$C_u(x, u)=(e , \sigma) \colon (x, u \colon c \to G(x)) \xrightarrow{(e, \sigma)} (x, u \colon c \to G(x))$$
where $e \colon x \to x$ is idempotent in $\D$ and $\sigma \colon (u \colon c \to G(x)) \to (u \colon c \to G(x))$ denotes the associated idempotent in $\C_{c/}$. Proceeding as in the proof of Proposition \ref{hypoinitial_vs_initial}(3), we consider the $\mathrm{Spine}_{\infty}^{\op}$-diagram in $G_{c/}$,
\[
E \colon \cdots \xrightarrow{(e, \sigma)}  (x, u) \xrightarrow{(e, \sigma )} (x,u) \xrightarrow{(e, \sigma)} (x,u),
\]
which admits a weak limit of order $0$, denoted by
\[
{}^{(y,v)} E \colon (\mathrm{Spine}_{\infty}^{\op})^{\triangleleft} \to G_{c/},
\]
where $(y, v \colon c \to G(y))\in G_{c/}$ denotes the cone object. Following Corollary \ref{limits_slices_under_objects}, this weak limit (of order $0$) is constructed by taking a weak limit of order $1$, denoted ${}^y \overline{E} \colon (\mathrm{Spine}_{\infty}^{\op})^{\triangleleft} \to \D$, of 
the induced diagram in $\D$,
$$\overline{E} \colon \mathrm{Spine}_{\infty}^{\op} \xrightarrow{E} G_{c/} \xrightarrow{q} \D,$$
with cone object $y \in \D$, and then lifting it appropriately to $G_{c/}$. 
The $0$-th/bottom component of the cone ${}^{(y,v)} E$ is a morphism
\[
(i, \tau) \colon (y, v) \to (x, u).
\]
Since $(e, \sigma)$ is an idempotent morphism, it follows that any other component of ${}^{(y,v)} E$ is canonically identified with $(i, \tau)$. In particular, $[i] = [e] \circ [i]$ in $\h(\D)$. We may postcompose $(i, \tau)$ with the cone $C_u$ to obtain a new cone on the identity $\id \colon G_{c/} \to G_{c/}$, denoted by 
\[
C_{v} \colon G_{c/}^{\triangleleft} \to G_{c/},
\]
whose cone object is $(y, v) \in G_{c/}$. In particular, $(y, v)$ is hypoinitial in $G_{c/}$. Therefore, using the characterization in Proposition \ref{hypoinitial_vs_initial}(1), it suffices to show that the component of the cone $C_{v}$ at $(y,v)$ 
$$C_{v}(y, v) \colon (y,v) \to (y,v)$$
is an equivalence. Equivalently, it suffices to show that the associated morphism in $\D$, 
$$\overline{C}_{v}(y) = q(C_{v}(y, v)) \colon y \to y,$$
is an equivalence. As in the proof of Proposition \ref{hypoinitial_vs_initial}(3), we consider a new cone ${}^{(x,u)} E$ on $E$, defined by 
$${}^{(x,u)} E \colon \Delta^0 \ast \mathrm{Spine}_{\infty}^{\op} \xrightarrow{\id \ast E} \Delta^0 \ast G_{c/} \xrightarrow{C_{u}} G_{c/}$$
whose cone object is $(x,u) \in G_{c/}$. 
There is a morphism of cones ${}^{(x,u)} E \to {}^{(y,v)} E$ defined by the following diagram:
$$\Delta^0 \ast (\Delta^0 \ast \mathrm{Spine}_{\infty}^{\op}) \xrightarrow{\id \ast {}^{(y,v)} E} \Delta^0 \ast G_{c/} \xrightarrow{C_u} G_{c/}.$$
Let $(r, \rho) \colon (x, u) \to (y, v)$ denote the associated morphism between the cone objects. By construction, this is just $C_{u}(y, v)$. In particular, we have $[e] = [i] \circ [r]$ in $\h(\D)$.

There is also a morphism of cones ${}^{(y,v)} E \to {}^{(x, u)} E$, which we may choose so that the associated morphism of cone objects is $(i, \tau) \colon (y, v) \to (x, u)$. To see this, note that precomposition of ${}^{(x,u)} E$ with $(i, \tau)$ defines a cone ${}^{(y, v)} E'$
with cone object $(y,v)$ as well as a morphism of cones ${}^{(y,v)} E' \to {}^{(x,u)} E$. In order to identify ${}^{(y,v)} E'$ with ${}^{(y,v)} E$, we observe that the space of cones on $E$ with cone object $(y, v)$ is given by the retract
$$\mathrm{lim}_{\mathrm{Spine}_{\infty}^{\op}} \map_{G_{c/}}((y, v), (x, u)) \subseteq \map_{G_{c/}}((y, v), (x, u)).$$
Thus, two such cones are equivalent if and only if their components are homotopic. This is satisfied by the pair of cones ${}^{(y,v)} E'$ and ${}^{(y,v)} E$.

The composite morphism of cones ${}^{(y,v)} E \to {}^{(x,u)} E \to {}^{(y,v)} E$ induces an endomorphism ${}^y \overline{E} \to {}^y \overline{E}$ of a weak $\N^{\op}$-limit of order $1$ in $\D$. Thus, the latter morphism is (non-canonically) equivalent to the identity morphism; in particular, $[r] \circ [i] = \id_y$.  
Combining these observations, we conclude the following identification of $\overline{C}_v(y)$:
$$[\overline{C}_v(y)] = [q(C_v(y,v))] = [q(r, \rho)] \circ [q(C_v(x,u))] = [r] \circ ( [e] \circ [i]) = [r] \circ [i] = \id_y.$$
In particular, $\overline{C}_v(y)$ is an equivalence, and therefore so is $C_v(y, v)$. This shows that $(y, v)$ is initial in $G_{c/}$ by the characterization in Proposition \ref{hypoinitial_vs_initial}(1). Thus, we have shown that $G_{c/}$ admits an initial object for every $c \in \C$, and so it follows that $G$ admits 
a left adjoint.
\end{proof}

\begin{example}
Theorem \ref{nGAFT} fails for $n=1$. For a counterexample, let $\h(\S)$ be the (ordinary) homotopy category of spaces and let $G \colon \h(\S)^{\op} \to \Set$ be a retract of a representable functor $\hom_{\h(\S)}(-, X)$ (\emph{hyporepresentable} in the sense of \cite{He}) which is not representable. Examples of such functors arise from idempotents $e \colon X \to X$ in $\h(\S)$ which do not split (see \cite{FH}). Then $G$ preserves small products and weak pullbacks (of order $0$) because these properties are closed under retracts (of functors) and representable functors have these properties. Moreover, $F$ satisfies the solution set condition. To see this, let $R \colon \hom_{\h(\S)}(-, X) \to G$ denote the retraction in the category of functors; then observe that for every set $S$, the object in $G_{S/}$ which is given by the composition of the map
$$S  \longrightarrow \hom_{\h(\S)}(\prod_{s \in S} X, X), \ \ s  \mapsto (\prod_{s \in S} X \xrightarrow{p_s} X)$$
followed by the map
$$R(\prod_{s \in S} X) \colon \hom_{\h(\S)}(\prod_{s \in S} X, X) \longrightarrow G(\prod_{s \in S} X)$$
defines a weakly initial object (in fact, this is also hypoinitial). Indeed, since there is a section $i \colon G \to \hom_{\h(\S)}(-, X)$, every object $(Y, h \colon S \to G(Y)) \in G_{S/}$ yields an associated collection of morphisms $\{(i(Y) \circ h)_s \colon Y \to X\}_{s \in S}$ and the induced canonical morphism $Y \to \prod_{s \in S} X$ defines a (natural) morphism to $(Y, h)$. Thus, $G$ satisfies the assumptions of Theorem \ref{nGAFT}, but $G$ does not admit a left adjoint since it is not representable. 
\end{example}

Following \cite{NRS}, an interesting application of these adjoint functor theorems is the following result showing that adjoint functors can often be detected 
on the (ordinary) homotopy category. The following theorem generalizes \cite[Theorem 3.3.1]{NRS}.

\begin{theorem} \label{homotopy_detect_adjoint}
Let $\D$ be a finitely weakly complete $n$-category and $\C$ an $n$-category, where $n \geq 1$ is an integer or $n=\infty$. Let $G \colon \D \to \C$ be a functor which preserves finite products and weak pullbacks of order $(n-1)$. Then $G$ admits a left adjoint if and only if $\h(G) \colon \h(\D) \to \h(\C)$ admits a left adjoint. 
\end{theorem}
\begin{proof}
We recall from Proposition \ref{characterization_adjoint_functor} that $G$ admits a left adjoint if and only if $G_{c/}$ admits an initial object for every $c \in \C$. 

Note first that if $G$ admits a left adjoint, then obviously so does $\h(G)$. 

Conversely, suppose that $\h(G)$ admits a left adjoint. The case $n=1$ is obvious, so we may assume that $n \geq 2$. By Corollary \ref{limits_slices_under_objects}, $G_{c/}$ has finite products and weak pullbacks of order $(n-2)$ for every $c \in \C$. By assumption, $\h(G)_{c/}$ has an initial object for every $c \in \C$. It is easy to see that the canonical functor 
$$\h(G_{c/}) \to \h(G)_{c/}$$
is surjective on objects, full, and conservative (see the proof of \cite[Theorem 3.3.1]{NRS}). 

Moreover, we claim that for every pair of morphisms in $\h(G_{c/})$,
$$f, g \colon (d, c \to G(d)) \rightrightarrows (d', c \to G(d')),$$
there is a morphism 
$$u_{f,g} \colon (d'', c \to G(d'')) \to (d, c \to G(d)) \text{ such that } f \circ u_{f, g} = g \circ u_{f,g}.$$ To see this, note first that $G_{c/}$ admits weak equalizers of order $(n-2)$. This is because cones on a diagram of the form $(\alpha, \beta \colon \ast \rightrightarrows \bullet)$ are identified with cones on the associated diagram $(\bullet \xrightarrow{\Delta} \bullet \times \bullet \xleftarrow{(\alpha, \beta)} \ast)$, and $G_{c/}$ admits products and weak pullbacks of order $(n-2)$ (or we may apply Proposition \ref{weakly_cocomplete_colimits} directly). By Proposition \cite[Proposition 3.20]{Ra}, the canonical functor $G_{c/} \to \h(G_{c/})$ sends these higher weak equalizers to weak equalizers of order $0$. It follows that $\h(G_{c/})$ admits weak equalizers (of order $0$), which shows our claim above about the existence of $u_{f,g}$. 

Applying \cite[Lemma 3.3.2]{NRS}, we conclude that the functor $\h(G_{c/}) \to \h(G)_{c/}$ preserves and detects initial objects. Thus, $G_{c/}$ has an 
$h$-initial object and then the result follows directly from Theorem \ref{nGAFTfin}.   
\end{proof}

\begin{example}
Let $\D$ be a complete $\infty$-category and consider the canonical functor $\gamma_n \colon \D \to \h_n\D$. By \cite[Corollary 3.22]{Ra}, the functor $\gamma_n$ preserves small products and weak pullbacks of order $(n-1)$. However, $\gamma_n$ does not admit a left adjoint unless $\D$ is equivalent to 
an $n$-category and $\gamma_n$ is an equivalence. To see this, note that if $\gamma_n$ happens to be a right adjoint, then it must preserve finite limits, and then the conclusion follows from \cite[Corollary 3.22]{Ra}. This example demonstrates the role of $n$ in the assumptions of Theorem \ref{homotopy_detect_adjoint}.
\end{example}

\begin{remark}
We mention the following weaker version of Theorem \ref{homotopy_detect_adjoint}, which can be obtained from Theorem \ref{nGAFT} instead of Theorem \ref{nGAFTfin} (cf. \cite[Remark 3.3.3]{NRS}). Let $\D$ be a locally small weakly complete $n$-category and $\C$ a 2-locally small $n$-category, where $n \geq 2$ is an integer or $n=\infty$. Let $G \colon \D \to \C$ be a functor which preserves small products and weak pullbacks of order $(n-1)$. 
Then $G$ admits a left adjoint if and only if $\h(G) \colon \h(\D) \to \h(\C)$ admits a left adjoint. For the proof, note that $G_{c/}$ admits 
a small weakly initial set if and only if $\h(G)_{c/}$ admits a small weakly initial set (see \cite[Proposition 3.2.2]{NRS}), and then apply Theorem \ref{nGAFT}. 
\end{remark}

Similarly to \cite[Corollary 3.3.5]{NRS}, we deduce the following corollary as a consequence of Theorem \ref{homotopy_detect_adjoint}.

\begin{corollary}
Let $\D$ be a finitely weakly complete $n$-category and $\C$ an $n$-category, where $n \geq 1$ is an integer or $n=\infty$. Let $G \colon \D \to \C$ be a functor which preserves finite products and weak pullbacks of order $(n-1)$. Then $G$ is an equivalence if and only if $\h(G) \colon \h(\D) \to \h(\C)$ is an equivalence.
\end{corollary}

\section{Higher Brown representability} \label{section:higher_Brown_rep}

\subsection{Preliminaries}

Let $\C$ be a locally small $n$-category, where $n \geq 1$ is an integer or $n=\infty$. We write $\map_{\C}(x, y)$ to denote a functorial model for the (small) mapping space of morphisms from $x$ to $y$ in $\C$ (see \cite{Ci, HTT}).  Note that this is an $(n-1)$-truncated $\infty$-groupoid \cite[Proposition 2.3.4.18]{HTT}.

We recall that $\S$ denotes the $\infty$-category of (small) spaces and $\Sn \subset \S$ the full subcategory of $\S$ that is spanned by the $(n-1)$-truncated spaces. 

\medskip

A functor $F \colon \C^{\op} \to \Sn$ is \emph{representable} if it is equivalent to a functor of the form $\map_{\C}(-, x) \colon \C^{\op} \to \Sn$ for some object $x \in \C$. Every representable functor satisfies the conditions (B1)--(B2) below.

\begin{definition} \label{Brown_rep_def}
Let $\C$ be a locally small weakly cocomplete $n$-category (see Definition \ref{weakly_cocomplete_def}), where $n \geq 1$ is an integer or $n=\infty$. We say that $\C$ \emph{satisfies Brown representability} if for any given functor $F \colon \C^{\op} \to \Sn$, $F$ is representable if (and only if) the following conditions are satisfied. 
\begin{enumerate}
\item[(B1).] For any small coproduct $\coprod_{i\in I} x_i$ in $\C$, the canonical morphism in $\Sn$
$$
F\left(\coprod_{i\in I} x_i\right)\longrightarrow \prod_{i\in I} F(x_i)
$$
is an equivalence.
\item[(B2).] For every weak pushout in $\C$ of order $(n-1)$
\begin{equation*} 
\xymatrix{
x \ar[r] \ar[d] & y \ar[d] \\ 
z \ar[r] & w}
\end{equation*}
the canonical morphism in $\Sn$
$$
F(w)\longrightarrow F(y)\times_{F(x)}F(z)
$$
is $(n-1)$-connected. 
\end{enumerate}
\end{definition}

We remark that the Brown representability context of Definition \ref{Brown_rep_def} is more general than the one in our previous work \cite[Section 5]{NRS} in two different respects. First, the present context aims to generalize Brown representability from suitable ordinary homotopy ($1$-)categories to higher homotopy ($n$-)categories, whereas the Brown representability context of \cite[Definition 5.1.1]{NRS} was concerned with ordinary categories which arise as homotopy $1$-categories of suitably nice $\infty$-categories. Second, we focus here directly on the abstract notion of a weakly cocomplete $n$-category, as a convenient axiomatization of the properties of homotopy $n$-categories of cocomplete $\infty$-categories, without the assumption of a model given by a  cocomplete $\infty$-category (cf. \cite[Definition 5.1.1]{NRS}).

\begin{remark}
If $I$ is the empty set, then property (B1) says that $F$ sends the initial object of $\C$ to a contractible space. Note that we have not assumed that $\C$ is pointed in general. If $\C$ happens to be pointed, then every functor $F$ satisfying (B1) is canonically pointed, i.e., it factors canonically through the category of pointed $(n-1)$-truncated spaces.
\end{remark}

The next proposition explains the connection between Brown representability and adjoint functor theorems (cf. \cite[Proposition 5.1.3]{NRS}). 

\begin{proposition} \label{Brown_implies_adjoints}
Let $\C$ and $\D$ be locally small $n$-categories, where $n \geq 1$ is an integer or $n=\infty$. Suppose that $\C$ is a weakly cocomplete $n$-category and satisfies Brown representability. Then a functor $F \colon \C \to \D$ admits a right adjoint if and only if $F$ satisfies the following properties:
\begin{itemize}
\item[(B1$'$).] $F$ preserves small coproducts. 
\item[(B2$'$).] $F$ preserves weak pushouts of order $(n-1)$. 
\end{itemize}
\end{proposition}
\begin{proof} The functor $F$ admits a right adjoint if and only if for every $d \in \D$, the associated functor 
$$Y_{F, d} \colon \C^{\op} \to \Sn, \   \ c \mapsto \map_{\D}(F(c), d),$$
is representable (see Proposition \ref{characterization_adjoint_functor}). Suppose that $F$ satisfies (B1$'$)--(B2$'$). Since $F$ satisfies (B1$'$), it follows that the above functor $Y_{F, d}$ 
satisfies (B1). Using Proposition \ref{characterization_weak_colimit}, $Y_{F, d}$ also satisfies (B2), because $F$ satisfies (B2$'$).  
Since $\C$ satisfies Brown representability by assumption, it follows that $Y_{F,d}$ is representable for any $d \in \D$, and therefore 
$F$ admits a right adjoint.

Conversely, suppose that $F$ is a left adjoint. Then it preserves coproducts and weak pushouts of order $(n-1)$ by Proposition \ref{adjoint_preserve_weak_co_limits}. 
\end{proof}

The last proposition has some special significance when $\C$ happens to admit all small colimits, especially, in the case $n=1$. 

\begin{corollary} \label{Brown_implies_adjoints_2}
Let $\C$ and $\D$ be locally small $n$-categories, where $n \geq 1$ is an integer or $n = \infty$. Suppose that $\C$ is cocomplete and satisfies Brown representability. Then the following statements for a functor $F \colon \C \to \D$ are equivalent:
\begin{enumerate}
\item[(a)] $F$ admits a right adjoint.
\item[(b)] $F$ preserves small colimits.
\item[(c)] $F$ satisfies (B1$'$)--(B2$'$). 
\end{enumerate}
\end{corollary}

The following corollary generalizes \cite[Corollary 5.1.4]{NRS}.

\begin{corollary} \label{Brown_implies_complete} 
Let $\C$ be a locally small cocomplete $\infty$-category such that either the associated weakly cocomplete $n$-category $\h_n \C$ satisfies Brown representability for some $n \geq 1$, or $\C$ itself satisfies Brown representability. Then $\C$ is complete.
\end{corollary}
\begin{proof}
Let $K$ be a small simplicial set and let $c \colon \C \to \C^K$ denote the constant $K$-diagram functor. We need to show that $c$ admits a right adjoint for any $K$. Since colimits in $\C^K$ are computed pointwise, it follows that $c$ preserves small colimits (see \cite[Corollary 5.1.2.3]{HTT} or \cite[Corollary 6.2.10]{Ci}). Moreover, $\C^K$ is again locally small (see \cite[Example 5.4.1.8]{HTT}). 

If $\C$ satisfies Brown representability, then $c$ admits a right adjoint by Proposition 
\ref{Brown_implies_adjoints} (for $n = \infty$).  

Suppose that $\h_n\C$ satisfies Brown representability for some $n \geq 1$. Since $c$ preserves small colimits, it follows that the functor $\h_n(c) \colon \h_n \C \to \h_n(\C^K)$ preserves small coproducts and weak pushouts of order $(n-1)$. This uses the fact that the canonical functor $\gamma_n \colon \C \to \h_n\C$ (resp. $\gamma_n \colon \C^K \to \h_n(\C^K)$) preserves small coproducts and sends pushouts to weak pushouts of order $(n-1)$ \cite[Proposition 3.20, Corollary 3.22]{Ra}. (In the special case $n=1$, note that it suffices to show that $\h(c)$ preserves \emph{some} choice of weak pushout for each diagram.) Then, applying Proposition \ref{Brown_implies_adjoints}, we conclude that $\h_n(c)$ admits a right adjoint, therefore $\h(c) \colon \h(\C) \to \h(\C^K)$ admits a right adjoint, too. Then the result follows from Theorem \ref{homotopy_detect_adjoint} applied to the functor $c \colon \C \to \C^K$ (see also \cite[Theorem 3.3.1 and Remark 3.3.3]{NRS}). 
\end{proof}


Similarly to the proof of Corollary \ref{Brown_implies_complete}, we may more generally combine Proposition \ref{Brown_implies_adjoints} and Theorem \ref{homotopy_detect_adjoint} and use \cite[Proposition 3.20]{Ra} to obtain the following (cf. \cite[Theorem 4.1.3, Corollary 5.1.5]{NRS}):

\begin{corollary} \label{Beown_implies_adjoints_3}
Let $\C$ and $\D$ be locally small $n$-categories, where $n \geq 1$ is an integer or $n=\infty$. Suppose that $\C$ is a weakly cocomplete $n$-category such that the weakly cocomplete $k$-category $\h_k \C$ satisfies Brown representability for some $1 \leq k \leq n$. Then a functor $F \colon \C \to \D$ admits a right adjoint 
if and only if $F$ preserves small coproducts and weak pushouts of order $(n-1)$.
\end{corollary}

\begin{remark}
Given a locally small weakly cocomplete $n$-category $\C$, note that a functor $F: \C^{\op} \to \Sn$ is representable if the functor $F^{\op} \colon \C \to \Sn^{\op}$ is a left adjoint. Assuming that $\C$ satisfies Brown representability, the converse holds for representable functors $\map_{\C}(-, x)$ which send weak pushouts of order $(n-1)$ to weak pullbacks of order $(n-1)$. (Note that weak pullbacks of order $(n-1)$ in $\Sn$ for $n \geq 2$ are just pullbacks in the complete $\infty$-category $\Sn$.)  
\end{remark}

\begin{example} \label{importance_of_n}
Let $\C$ be a locally small cocomplete $\infty$-category. The canonical functor $\gamma_n \colon \C \to \h_n \C$ preserves small coproducts and weak  pushouts of order $(n-1)$ (see \cite[Proposition 3.20, Corollary 3.22]{Ra}). However, $\gamma_n$ does not admit a right adjoint in general -- even if $\C$ satisfies Brown representability. This example demonstrates the importance of the double function of $n$ in Proposition \ref{Brown_implies_adjoints}. 
\end{example}

Our main results in this section will show examples of $n$-categories which satisfy Brown representability; these are discussed in the following subsections. We note first the following elementary inheritance property of Brown representability that will allow us to generate new examples from old ones. 

We recall that a functor $L \colon \C \to \D$ between $\infty$-categories is a \emph{(Bousfield) localization} of $\C$ if $L$ admits a right adjoint which is fully faithful \cite[5.2.7]{HTT}. We also say that $\D$ is a \emph{(Bousfield) localization} of $\C$ if there is a localization functor $L \colon \C \to \D$. 

\begin{proposition} \label{local_preserves_Brown}
Let $\C$ be a locally small weakly cocomplete $n$-category that satisfies Brown representability and let $\D$ be an $n$-category. Suppose that $L \colon \C \to \D$ is a localization of $\C$. Then $\D$ is a locally small weakly cocomplete $n$-category and satisfies Brown representability.
\end{proposition}
\begin{proof}
We may assume that $i \colon \D \subset \C$ is a full subcategory and $L \colon \C \to \D$ is the left adjoint to the inclusion functor. It is clear that $\D$ is again a locally small $n$-category. Moreover, since a weak colimit in $\D$ can be computed by applying the left adjoint $L$ to a weak colimit in $\C$ (Proposition \ref{adjoint_preserve_weak_co_limits}), it follows that $\D$ is also weakly $n$-cocomplete. Therefore it remains to prove that $\D$ satisfies Brown representability. Given a functor $F \colon \D^{op} \to \Sn$ that satisfies (B1)--(B2), the composite functor $\C^{\op} \stackrel{L^{\op}}{\to} \D^{op} \stackrel{F}{\to} \Sn$ also satisfies (B1)--(B2) (again by Proposition \ref{adjoint_preserve_weak_co_limits}). Since $\C$ satisfies Brown representability by assumption, the functor $F \circ L^{\op}$ is representable by some object $x \in \C$. Moreover, the composite functor $F \circ L^{\op}$ obviously sends $L$-equivalences to equivalences. This implies that $x$ is $L$-local, therefore $x \simeq i L(x)$. Thus, $F$ is representable by the object $L(x) \in \D$.   
\end{proof}

\subsection{Compactly generated $n$-categories} \label{subsec:compactly_generated} In this subsection, we introduce the class of \emph{compactly generated n-categories} 
and prove that these satisfy Brown representability. 
Compactly generated $n$-categories define a convenient general context for a Brown representability theorem for abstract $n$-categories, that is, for $n$-categories which do not necessarily arise as homotopy $n$-categories. The definition of a compactly generated $n$-category is inspired by related definitions from the various classical Brown representability contexts 
for ordinary categories \cite{He, Ne} or for homotopy ($1$-)categories of $\infty$-categories with special properties \cite{HA, NRS}.  

\begin{definition} \label{weak_gen_def}
Let $\C$ be a locally small $\infty$-category. A set of \emph{weak generators} of $\C$ is a (small) set of objects $\mathcal{G}$ that jointly detect equivalences, i.e., a morphism $f \colon x\to y$ in $\C$ is an equivalence if and only if the 
canonical morphism
$$
\xymatrix{
\map_{\C}(g,x) \ar[r]^{f_*} & \map_{\C}(g,y)
}
$$
is an equivalence for every object $g \in \mathcal{G}$. 
\end{definition}

\begin{remark} 
Some expository comments about the terminology may be useful here. First, weak generators should not be confused with the strictly stronger notion which refers to a set of objects that generate $\C$ under filtered colimits. 

Second, in the case of ordinary categories, it differs also from the familiar notion of a set of objects which distinguish parallel arrows. The latter property follows from the property in Definition \ref{weak_gen_def} when $\C$ admits equalizers.

Third, the related notion of a set of objects that detect whether the canonical morphism to the terminal object $(x \to *)$ is an equivalence is strictly weaker in general. For example, the collection of spheres $\{S^n\}_{n \geq 0}$ is not a set of weak generators in the 
homotopy category of spaces \cite{He, AC}, but the spheres obviously detect whether a space is weakly contractible. These two properties of a set of objects are equivalent for the homotopy $n$-category of a stable $\infty$-category. This uses the fact that a morphism in a stable $\infty$-category is an equivalence if and only if its (co)fiber is a zero object. 

Finally, the definition of weak generators in an $\infty$-category $\C$ used in \cite[Definition 5.2.1]{NRS} corresponds to the definition of weak generators in $\h(\C)$ in the sense of Definition \ref{weak_gen_def}. We recall that the Brown representability context of \cite{NRS} presupposed an underlying locally small cocomplete $\infty$-category, whereas now we do not assume that our candidate locally small $n$-category $\C$ arises as the homotopy $n$-category of a locally small cocomplete $\infty$-category.
\end{remark}

\begin{proposition} \label{weak_gen_characterization}
Let $\C$ be a locally small $\infty$-category which admits finite colimits. The following are equivalent:
\begin{itemize}
\item[(a)] $\C$ has a set of weak generators. 
\item[(b)] $\h_n \C$ has a set of weak generators for some $n \geq 1$.
\item[(c)] $\h_n \C$ has a set of weak generators for every $n \geq 2$.
\item[(d)] $\h_2 \C$ has a set of weak generators.
\end{itemize}
If $\C$ is stable, then (a)--(d) are also equivalent to:
\begin{itemize}
\item[(e)] $\h(\C)$ has a set of weak generators.
\end{itemize} 
\end{proposition}
\begin{proof}
The implications (e) $\Rightarrow$ (d) $\Rightarrow$ (c) $\Rightarrow$ (b) $\Rightarrow$ (a) are obvious. (a) $\Rightarrow$ (d): Let $\mathcal{G}$ denote a set of weak generators in $\C$. We may assume that $\mathcal{G}$ is closed under tensoring of objects with $S^k$ for any $k \geq 0$. Let $f \colon x \to y$ be a morphism in $\C$ such that the map
$$
\map_{\h_2 \C}(g,x) \xrightarrow{f_*} \map_{\h_2 \C}(g,y)
$$
is an equivalence for every object $g \in \mathcal{G}$. Since $\mathcal{G}$ is closed under tensoring with $S^k$, it follows that the (horizontal) map 
$$
\xymatrix{
\map_{\h_2 \C}(g \otimes S^k, x) \ar[r]^{f_*}_{\simeq} \ar[d]_{\simeq} & \map_{\h_2 \C}(g \otimes S^k, y) \ar[d]^{\simeq} \\
\map_{\h_2\S}(S^k, \map_{\C}(g, x)) \ar[r]^{f_*} & \map_{\h_2\S}(S^k, \map_{\C}(g, y)) 
}
$$
is an equivalence for any $g \in \mathcal{G}$ and $k \geq 0$. (Here we have used the notation $S^k$ for the associated $\infty$-groupoid in $\S$.) The collection of spheres $\{S^k\}_{k \geq 0}$ defines a set of weak generators in $\h_2 \S$ \cite{AC}. So the map
$$
\map_{\C}(g,x) \xrightarrow{f_*} \map_{\C}(g,y)
$$
is an equivalence for every object $g \in \mathcal{G}$, and this implies that $f$ is an equivalence. This completes the proof of (a) $\Rightarrow$ (d). 
An analogous argument shows (a) $\Rightarrow$ (e): if $\C$ is stable, the spheres $\{S^k\}_{k \geq 0}$ detect already in $\h(\S)$ whether the map 
of (infinite loop) spaces 
$$
\map_{\C}(g,x) \xrightarrow{f_*} \map_{\C}(g,y)
$$
is an equivalence, because the components of the mapping spaces are simple in this case. 
\end{proof}

Next we consider a flexible and general notion of compactness for objects in an $\infty$-category. Recall that $\N$ denotes the $\infty$-category associated with the poset of non-negative integers with its usual ordering. 

\begin{definition} \label{compact} 
Let $\C$ be a locally small $\infty$-category which admits weak $\N$-colimits of order $(-1)$ and suppose that every diagram $T \colon \N \to \C$ is equipped with a \emph{distinguished} cone $T^{\triangleright} \colon \N^{\triangleright} \to \C$ with cone object $\colim^w T$. An object $x \in \C$ is \emph{compact} (with respect to these distinguished cones) if for every diagram $T \colon \N \to \C$,  the canonical map
$$
\colim_{i \in \N}\map_{\C}(x, T(i))\to \map_{\C}(x, \colim^w T)
$$
is an equivalence.
\end{definition}

The distinguished cones will normally be given by weak colimits of order $t \geq 0$. This choice is unique up to (non-canonical) equivalence when $t >0$. For $t=0$, compactness is to be understood as in \cite{He}, that is, it is defined in terms of weak $\N$-colimits $T^{\triangleright}$ (of order $0$) which have been chosen and fixed once and for all in advance. Definition \ref{compact} generalizes this notion to $\infty$-categories which are equipped with a choice of a \emph{distinguished} cone for each $\N$-diagram. 

\begin{example}  \label{finitely_pres_objects}
Let $\C$ be a locally small $\infty$-category which admits filtered colimits. An object $x \in \C$ is compact (in the sense of Definition \ref{compact}) 
if $x \in \C$ is finitely presentable (also called compact in \cite{HTT}), that is, if the representable functor $\map_{\C}(x, -)$ preserves filtered colimits. In this case, the distinguished cones of Definition \ref{compact} are taken to be the $\N$-colimits. In addition, $x \in \h_n \C$ is also compact (in the sense of Definition \ref{compact}), where the distiguished cones are taken to be the weak $\N$-colimits (of order $(n-1)$) that arise from lifting each $\N$-diagram to $\C$ and taking the $\N$-colimit in $\C$. 
\end{example}

\begin{remark}
The above definition of compactness differs from the notion of $h$-compactness in \cite[5.2]{NRS} which applied to locally small $\infty$-categories with $\N$-colimits and their homotopy ($1$-)categories. The present definition generalizes to $\infty$-categories the corresponding definition in \cite{He} (restricted here to 
$\N$-diagrams).
\end{remark}  

Note that this notion of compactness can be used in the context of weakly cocomplete $n$-categories because these admit weak $\N$-colimits of order $(n-1)$. In particular, these weak colimits are unique up to (non-canonical) equivalence when $n>1$. (To see this, first recall that the ordinary category $\N$ is Joyal equivalent to the one-dimensional simplicial set given by its spine $\mathrm{Spine}_{\infty}$, and then conclude by applying \cite[Proposition 3.10]{Ra} to the case of the skeletal decomposition of $\mathrm{Spine}_{\infty}$ or by applying directly Proposition \ref{weakly_cocomplete_colimits}.) In fact, this is the main reason why we restrict to  compactness only with respect to $\N$-diagrams instead of considering more general diagrams indexed by larger ordinals (cf. \cite{He}). Thus, we may make the following assumption concerning compactness in a weakly cocomplete $n$-category.

\smallskip

\noindent \textbf{Assumption.} When $n>1$, compactness in a weakly cocomplete $n$-category is defined unambiguously with respect to the weak $\N$-colimits of order $(n-1)$. When $n=1$, the notion will generally depend on a choice of weak $\N$-colimits of order $0$, which is tacitly assumed.

\begin{definition}(\emph{Compactly generated $n$-category}) \label{compactly_gen_def} Let $n \geq 1$ be an integer or $n=\infty$. 
A locally small $n$-category $\C$ is called \emph{compactly generated} if $\C$ is a weakly cocomplete $n$-category and has a set of weak generators $\mathcal{G}$ consisting of compact objects. 
\end{definition}

\begin{remark} \label{clarify_terms}
The terminology of Definition \ref{compactly_gen_def} slightly clashes with our definition of a compactly generated $\infty$-category in \cite{NRS}. The definition in \cite[5.2]{NRS} specifies a class of locally small cocomplete $\infty$-categories in terms of the properties of their homotopy categories, whereas Definition \ref{compactly_gen_def} does not presuppose that $\C$ arises as the homotopy $n$-category of a cocomplete $\infty$-category. The two definitions are related simply as follows: given a compactly generated $\infty$-category $\C$ in the sense of \cite{NRS}, then $\h(\C)$ is a compactly generated $1$-category in the sense of Definition \ref{compactly_gen_def}.
\end{remark}

Here are some of the main examples of compactly generated $n$-categories. 


\begin{example} \label{spaces_example}
A finitely presentable $\infty$-category admits a set of weak generators given by the finitely presentable objects. Thus, for any $n \geq 2$, the homotopy $n$-category of a finitely presentable $\infty$-category is compactly generated (as an $n$-category). This follows from Proposition \ref{weak_gen_characterization} and Examples \ref{weakly_cocomplete_example} and \ref{finitely_pres_objects}. Note that this class of examples includes also the ordinary locally finitely presentable categories. Also, for any $n \geq 2$, the homotopy $n$-category of the $\infty$-category of spaces $\mathcal{S}$ is compactly generated and the set of spheres $\{S^k\}_{k \geq 0}$ defines a set of compact weak generators (see \cite{AC}). More generally, for any small simplicial set $K$ and $n \geq 2$, the homotopy $n$-category of the $\infty$-category $\mathcal{S}^K$ is also a compactly generated $n$-category. 
\end{example}

\begin{example} \label{spectra_example}
The homotopy $n$-category of a finitely presentable stable $\infty$-category is a compactly generated $n$-category for any $n \geq 1$ -- again this follows from Proposition \ref{weak_gen_characterization} and Examples \ref{weakly_cocomplete_example} and \ref{finitely_pres_objects}. In particular, the homotopy $n$-category of the  stable $\infty$-category of spectra $\mathcal{S}p$ is compactly generated for any $n \geq 1$. More generally, for any small simplicial set $K$ and $n \geq 1$, the homotopy $n$-category of the $\infty$-category $\mathcal{S}p^K$ is also a compactly generated $n$-category  (cf. \cite[Example 5.2.5]{NRS}). 
\end{example}

The following result is our main Brown representability theorem in this section. This result is a generalization to $n$-categories of Heller's Brown representability theorem for ordinary categories \cite{He}, restricted to the compactly generated case.  It also generalizes the Brown representability 
theorem for $(2,1)$-categories (or categories enriched in groupoids) that was proved in \cite{Ca}. 

\begin{theorem} \label{higher_Brown}
Let $\C$ be a compactly generated $n$-category, where $n \geq 1$ is an integer or $n=\infty$. Then $\C$ satisfies Brown representability.
\end{theorem}

The proof of this theorem is based on (the dual of the) following criterion for the existence of terminal objects. The proof of this criterion is somewhat long and technical, but it is essentially a generalization of familiar arguments from the classical proof of the Brown representability theorem.  

\begin{lemma}[Criterion C] \label{criterion_C}
Let $\C$ be a locally small $n$-category, where $n \geq 1$ is an integer or $n=\infty$. Suppose that $\C$ admits small coproducts and weak pushouts  
of order $(n-2)$. In addition, suppose that $\C$ has a set of weak generators $\mathcal{G}$ which consists of compact objects (with respect to a choice of distinguished cones as in Definition \ref{compact}). Then $\C$ 
has a terminal object. 
\end{lemma}
\begin{proof}
For each $g \in \mathcal{G}$ and $0 \leq m \leq n$, let $c^{\partial}_g \colon \partial \Delta^m \to \C$ and $c_g \colon \Delta^m \to \C$ denote the constant diagrams at $g \in \C$. By our assumptions on $\C$ and Proposition \ref{weakly_cocomplete_colimits}, there is a weak colimit $C^{\partial}_g$ (resp. $C_g$) of $c^{\partial}_g$ (resp. $c_g$) of order $(n-2)-(m-1) + 1 = (n-m)$, with cone object denoted by $g \otimes \partial \Delta^m$ (resp. $g \otimes \Delta^m$). We may choose $C_g$ to be again 
the constant diagram, in which case $g \otimes \Delta^m$ is equivalent to $g$ and $C_g$ is a colimit of $c_g$. The  cone $C^{\partial}_g$ is unique up to (non-canonical) equivalence for $m < n$. In addition, there are morphisms $C^{\partial}_g \to C_g$ (as cones on diagrams defined on $\partial \Delta^m$), which restrict to morphisms of cone objects, denoted $i^m_g \colon g \otimes \partial \Delta^m \to g \otimes \Delta^m$. These morphisms exist -- and are unique up to homotopy when $m < n$ --, since $C^{\partial}_g$ is a weak colimit of order $\geq 0$ for any $m \leq n$. 

\smallskip

Let $S$ denote the set of morphisms $i_g^m$ for all $g \in \mathcal{G}$ and $0 \leq m \leq n$. 

\smallskip

\noindent \emph{Step 1: Existence of $\mathcal{G}$-terminal objects.} We say that an object $x \in \C$ is $\mathcal{G}$-\emph{terminal} if $\map_{\C}(g, x)$ is contractible for every $g \in \mathcal{G}$. Our first goal is to prove that every object $c \in \C$ admits a morphism $(u_c \colon c \to x)$ to a $\mathcal{G}$-terminal object $x \in \C$. The construction of the morphism $u_c$ is essentially based on a small object argument with respect to the set of morphisms $S$. More precisely, given an object $c \in \C$, we will construct a diagram $x_{\bullet} \colon \N \to \C$ inductively as follows. Set $x_0 = c$. Assuming that $x_{\bullet}$ has been constructed for $\bullet < k$, we define $x_{k-1} \to x_k$ by considering a weak pushout of order $(n-2)$ (this is unique up to equivalence when $n > 2$):
$$
\xymatrix{
\coprod_{m \geq 0} \coprod_{g \in \mathcal{G}} \coprod_{T_{m, g}}  g \otimes \partial \Delta^m \ar[d] \ar[r] & x_{k-1} \ar[d] \\
\coprod_{m \geq 0} \coprod_{g \in \mathcal{G}} \coprod_{T_{m,g}} g \otimes \Delta^m \ar[r] & x_k
}
$$
where $T_{m,g}$ denotes a set of morphisms $(g \otimes \partial \Delta^m \to x_{k-1})$, one from each homotopy class. The top morphism is defined by these morphisms in the obvious way -- this uses that $\C$ admits small coproducts. This morphism suffices in order to extend the diagram $x_{\bullet}$ to all $\bullet \leq k$. Therefore, we obtain by induction the required diagram $x_{\bullet} \colon \N \to \C$. Let $\overline{x}_{\bullet} \colon \N^{\triangleright} \to \C$ be the distinguished cone on $x_{\bullet}$ (Definition \ref{compact}) and let $x_{\infty}$ denote the cone object. Then it suffices to prove that $x_{\infty}$ is $\mathcal{G}$-terminal. By construction, the mapping space $\map_{\C}(g, x_{\infty})$ is non-empty for every $g \in \mathcal{G}$. Therefore, it suffices to prove 
that for every $g \in \mathcal{G}$ and $m \leq n$, any map $\partial \Delta^m \to \map_{\C}(g, x_{\infty})$ is homotopically constant. Note that there is an 
equivalence
$$\colim_{i \in \N} \map_{\C}(g, x_i) \stackrel{\simeq}{\to} \map_{\C}(g, x_{\infty})$$
since $g$ is compact with respect to the distinguished cone $\overline{x}_{\bullet}$. Therefore, a map $$a \colon \partial \Delta^m \to \map_{\C}(g, x_{\infty})$$ factors up to homotopy through a map 
$a_i \colon \partial \Delta^m \to \map_{\C}(g, x_i)$. Since $g \otimes \partial \Delta^m$ is a weak colimit of order $(n-m)$, the map $a_i$ is induced (up to homotopy) by a morphism $\widetilde{a}_i \colon g \otimes \partial \Delta^m \to x_i$ in $T_{m,g}$ -- this morphism $\widetilde{a}_i$ is unique up to homotopy when $m < n$ (see Proposition \ref{characterization_weak_colimit}).  By construction, the composition 
$$g \otimes \partial \Delta^m \stackrel{\widetilde{a}_i}{\to} x_i \to x_{i+1}$$
factors (in a preferred way) through $i_g^m \colon g \otimes \partial \Delta^m \to g \otimes \Delta^m \simeq g$. As a consequence, the composition 
$$\partial \Delta^m \stackrel{a_i}{\to} \map_{\C}(g, x_i) \to \map_{\C}(g, x_{i+1})$$
factors (in a preferred way) through $\partial \Delta^m \subseteq \Delta^m \simeq \ast$. This shows that $x_{\infty}$ is indeed $\mathcal{G}$-terminal, 
and completes the construction of the morphisms $u_c$, for any $c \in \C$. In particular, the argument shows the existence of $\mathcal{G}$-terminal objects in $\C$ (e.g. by taking $c$ to be the initial object of $\C$). 

\smallskip

\noindent \emph{Step 2: Uniqueness of $\mathcal{G}$-terminal objects.} We first note that every morphism $x \to y$ between $\mathcal{G}$-terminal objects in $\C$ is an equivalence. This is an obvious consequence of the fact that $\mathcal{G}$ is a set of weak generators. Now suppose that $x$ and 
$y$ are $\mathcal{G}$-terminal objects in $\C$. Consider the coproduct $z : = x \sqcup y$ in $\C$ and the morphism $u_z \colon z \to w$ to a $\mathcal{G}$-terminal object $w \in \C$, as constructed in Step 1. Then the composite morphisms $(x \to z \to w)$ and $(y \to z \to w)$ are morphisms 
between $\mathcal{G}$-terminal objects, therefore they are equivalences. This shows that $x$ and $y$ are indeed equivalent in $\C$. 

\smallskip 

\noindent \emph{Step 3: $\mathcal{G}$-terminal objects are terminal.} Let $x \in \C$ be a $\mathcal{G}$-terminal object. We claim that the mapping space 
$\map_{\C}(c, x)$ is contractible for any $c \in \C$. Note that the mapping space $\map_{\C}(c, x)$ is non-empty as a consequence of Steps 1 and 2. Then it suffices to show that every map $a \colon \partial \Delta^m \to \map_{\C}(c, x)$, where $m \leq n$, is homotopically constant. It will be convenient to use the model for the mapping space $\map_{\C}(c,x)$ defined by the pullback square:
$$
\xymatrix{
\map_{\C}(c, x) \ar[r] \ar[d] & \C^{\Delta^1} \ar[d] \\ 
\Delta^0 \ar[r]^{(c, x)} & \C \times \C.
}
$$ 
(This has the correct homotopy type when $\C$ is an $\infty$-category by \cite[Corollary 4.2.1.8]{HTT}.) 

Then we may assume that the map $a$ arises from a diagram $\widetilde{a} \colon K \to \C$, where $K$ is defined by the pushout of simplicial sets
$$
\xymatrix{
\partial \Delta^m \times \partial \Delta^1 \ar[r] \ar[d] &  \partial \Delta^m \times \Delta^1 \ar[d] \\
\Delta^0 \times \partial \Delta^1 \ar[r] & K.
}
$$
Note that $K$ has two $0$-simplices, which we denote by $0$ and $1$, and the diagram $\widetilde{a}$ sends $0$ to $c \in \C$ and $1$ to $x \in \C$. The diagram $\widetilde{a}$ admits a weak colimit $\widetilde{a}^{\triangleright} \colon K^{\triangleright} \to \C$ of order $(n-1) - \mathrm{dim}(K) = n-m-1 \geq -1$ (see Proposition \ref{weakly_cocomplete_colimits}). Let $z$ denote the cone object and let $h \colon x \to z$ denote the corresponding morphism to the cone object. Consider a morphism $u_z \colon z \to w$ to a $\mathcal{G}$-terminal object, as constructed in Step 1. By the arguments in Step 2, the composition $(f \colon x \xrightarrow{h} z \xrightarrow{u_z} w)$ is an equivalence. Thus, the map $a$ is homotopically constant if and only if the composition 
$$\partial \Delta^m \stackrel{a}{\to} \map_{\C}(c, x) \xrightarrow{f \circ -} \map_{\C}(c, w)$$
is homotopically constant. Replacing $K$ by an $\infty$-category if so desired, this composite diagram can be identified with the composition
\begin{equation} \label{composition} \tag{*}
\partial \Delta^m \to \map_K(0,1) \stackrel{\widetilde{a}}{\to} \map_{\C}(c, x) \xrightarrow{f \circ -} \map_{\C}(c, w).
\end{equation}
Further, assuming that $K$ is replaced by an $\infty$-category (denoted also by $K$), we see that since $\widetilde{a}^{\triangleright}$ extends $\widetilde{a}$, the composition with the morphism to the cone object $(h \colon x \to z)$ yields a homotopy commutative diagram of spaces:
$$
\xymatrix{
\map_K(0,1) \ar[r]^{\widetilde{a}} \ar[d]_{(1 \to \ast) \circ -} & \map_{\C}(c, x) \ar[d]^{(h \colon x \to z) \circ -} \\
\map_{K^{\triangleright}}(0, *) \ar[r]^{\widetilde{a}^{\triangleright}} & \map_{\C}(c, z). 
}
$$
So the composition \eqref{composition} factors up to homotopy through the contractible mapping space 
$\map_{K^{\triangleright}}(0, *) \simeq \ast$, therefore, the composition \eqref{composition} is homotopically constant. 
This completes the proof that $x$ is terminal in $\C$. 
\end{proof}

\noindent \textbf{Proof of Theorem \ref{higher_Brown}.} Let $F \colon \C^{\op} \to \Sn$ be a functor which satisfies (B1)--(B2). Then it suffices to show that the $\infty$-category $F_{*/}$ admits an initial object; indeed this property is a necessary and sufficient criterion for the representability of $F$ by Proposition \ref{characterization_adjoint_functor}(2)$\Leftrightarrow$(3). Note that $F_{*/}$ is locally small and equivalent to an $n$-category. Moreover, $F_{*/}$ admits small products by Corollary \ref{limits_slices_under_objects}. It suffices then to show that $F_{*/}$ satisfies the rest of the assumptions of (the dual of) Lemma \ref{criterion_C} (Criterion C). The desired result would then follow directly by applying that lemma.  

We show first that $F_{*/}$ admits weak pullbacks of order $(n-2)$ -- this does not follow from Corollary \ref{limits_slices_under_objects} because $F$ does not preserve weak pullbacks of order $(n-1)$ in general! To see the existence of these weak pullbacks, consider a diagram in $F_{*/}$ depicted as follows: 
$$
\xymatrix{
& (c_2, x_2 \in F(c_2)) \ar[d] \\
(c_1, x_1 \in F(c_1)) \ar[r] & (c_0, x_0 \in F(c_0)) 
}
$$
where the notation $(c, x \in F(c))$ refers to the object $(c, * \xrightarrow{x} F(c))$. More specifically, this diagram consists of a diagram in $\C$
$$
\xymatrix{
c_0 \ar[r] \ar[d] & c_2 \\
c_1 &
}
$$
together with a square/cone in $\Sn$
$$
\xymatrix{
\Delta^0 \ar[r]^{x_2} \ar[d]_{x_1} \ar[dr]^{x_0} & F(c_2) \ar[d] \\
F(c_1) \ar[r] & F(c_0).
}
$$
We may form a weak pushout of order $(n-1)$ in $\C$:
$$
\xymatrix{
c_0 \ar[r] \ar[d] & c_2 \ar[d] \\
c_1 \ar[r] & c
}
$$
and lift the canonical map $\Delta^0 \to F(c_1) \times_{F(c_0)} F(c_2)$ along the $(n-1)$-connected map $F(c) \to F(c_1) \times_{F(c_0)} F(c_2)$. We obtain in this way a cone on our original diagram in $F_{*/}$
$$
\xymatrix{
(c, x \in F(c)) \ar[r] \ar[d] & (c_2, x_2 \in F(c_2)) \ar[d] \\
(c_1, x_1 \in F(c_1)) \ar[r] & (c_0, x_0 \in F(c_0)).
}
$$
We claim that this square is a weak pullback of order $(n-2)$. By Proposition \ref{characterization_weak_colimit}, this holds if for any $(c', x' \in F(c'))$, the
canonical map of mapping spaces in $F_{*/}$
\smallskip
\begin{equation} \label{Brown_comparison_map} \tag{*}
\small{
\map((c', x'), (c, x)) \to \map((c', x'), (c_1, x_1)) \times_{\map((c', x'), (c_0, x_0))} \map((c', x'), (c_2, x_2))
}
\end{equation}
\smallskip
is $(n-2)$-connected. Using that $F_{*/}$ is a (homotopy) pullback of $\infty$-categories, we have an identification of its mapping spaces,
$$\map((c', x'), (c, x)) \simeq \map_{\C}(c, c') \times_{\map(F(c'), F(c))} \map((F(c'), x'), (F(c), x)),$$
and similarly for the other mapping spaces in \eqref{Brown_comparison_map}. Thus, we have a pullback square of spaces
\begin{equation} \label{Brown_pullback_square} \tag{**}
\xymatrix{
\map((c', x'), (c,x)) \ar[d] \ar[r] & \map(c, c') \ar[d] \\
\Delta^0 \ar[r]^x & F(c)
}
\end{equation}
where the right vertical map is given by applying $F$ and evaluating at $x' \in F(c')$. There are of course similar pullbacks for the other 3 terms in \eqref{Brown_comparison_map}. Therefore we may identify \eqref{Brown_comparison_map} with the induced map between vertical fibers (over $x \in F(c)$) in the following square:
$$
\xymatrix{
\map(c, c') \ar[rr] \ar[d] && \map(c_1, c') \times_{\map(c_0, c')} \map(c_2, c') \ar[d] \\
F(c) \ar[rr] && F(c_1) \times_{F(c_0)} F(c_2). 
}
$$
Since both horizontal maps are $(n-1)$-connected, it follows that the induced map between the fibers is $(n-2)$-connected, as required. This shows that $F_{*/}$ admits weak pullbacks of order $(n-2)$. 

We show next that the opposite $\infty$-category of $F_{*/}$ admits a set of weak generators. Consider the following set of objects in $F_{*/}$, 
$$\mathcal{G}' = \{ (g, x \in F(g)) \ | \  g \in \mathcal{G} \text{ and } x \colon * \to F(g) \},$$
where $\mathcal{G}$ denotes a set of weak generators in $\C$. Then it is easy to see using \eqref{Brown_pullback_square} that $\mathcal{G}'$ defines a set of weak generators in the opposite $\infty$-category of 
$F_{*/}$, since $\mathcal{G}$ is a set of weak generators in $\C$ by assumption. 

Lastly, we claim that the objects in $\mathcal{G}'$ are compact in the opposite of $F_{*/}$. To see this, we must first clarify the relevant choice of distinguished $\N$-cones in the opposite of $F_{*/}$. Let $T \colon \N \to (F_{*/})^{\op}$ be an $\N$-diagram, depicted as follows
$$(t_0, x_0) \to (t_1, x_1) \to \cdots \to (t_n, x_n) \to \cdots$$
The composite diagram $\N \xrightarrow{T} (F_{*/})^{\op} \rightarrow \C$ has a distinguished cone $T^{\triangleright}_{\C} \colon \N^{\triangleright} \to \C$, which is given by a (distinguished) weak $\N$-colimit of order $(n-1)$:
$$t_0 \to t_1 \to \cdots \to t_n \to \cdots \to t_{\infty}.$$ 
Equivalently, the weak $\N$-colimit $T^{\triangleright}_{\C}$ is determined by a weak pushout in $\C$ of order $(n-1)$ (`telescope construction'):
\begin{equation} \label{telescope} \tag{***}
\xymatrix{
\bigsqcup_{i \in \N} (t_i \sqcup t_i) \ar[r] \ar[d] & \bigsqcup_{i \in \N} t_i \ar[d] \\ 
\bigsqcup_{i \in \N} t_i \ar[r] & t_{\infty}.
}
\end{equation}
Consider the induced tower in $\Sn$ which is obtained after applying $F$, 
$$F(t_0) \leftarrow F(t_1) \leftarrow \cdots \leftarrow F(t_n) \leftarrow \cdots \leftarrow F(t_{\infty}).$$
Since $F$ satisfies (B1) and (B2), it follows from \eqref{telescope} that the canonical map of spaces
$$F(t_{\infty}) \to \mathrm{lim}_{i \in \N^{\op}} F(t_i)$$
is $(n-1)$-connected. In particular, it is $0$-connected, so we may choose a point $x_{\infty} \in F(t_{\infty})$ and extend the tower of pointed spaces induced by $T$, 
$$(F(t_0), x_0) \leftarrow (F(t_1), x_1) \leftarrow \cdots \leftarrow (F(t_n), x_n) \leftarrow \cdots,$$
to a cone of pointed spaces   
$$(F(t_0), x_0) \leftarrow (F(t_1), x_1) \leftarrow \cdots \leftarrow (F(t_n), x_n) \leftarrow \cdots (F(t_{\infty}), x_{\infty})$$
that simultaneously lifts (the opposite of) $T^{\triangleright}_{\C}$. This cone and the cone $T^{\triangleright}_{\C}$ determine a distinguished cone in $(F_{*/})^{\op}$ on the $\N$-diagram $T$. This process defines the choice of distinguished cones on $\N$-diagrams in the opposite $\infty$-category of $F_{*/}$. 

It is easy to see that the objects of $\mathcal{G}'$ are compact (with respect to the distinguished $\N$-cones constructed above) by using the description of the mapping spaces in \eqref{Brown_pullback_square} and the assumption that $\mathcal{G}$ consists of compact objects in $\C$. Indeed, a distingushed 
$\N$-cone in the opposite of $F_{*/}$ consists of a distinguished weak $\N$-colimit in $\C$ of order $(n-1)$, 
$$t_0 \to t_1 \to \cdots \to t_n \to \dots \to t_{\infty},$$
together with a compatible sequence of points $x_i \in F(t_i)$ for all $i=0, 1, \ldots, \infty$. To see that an object $(g, x) \in \mathcal{G}'$ is compact in 
$(F_{*/})^{\mathrm{op}}$, we need to show that the canonical map involving mapping spaces in $F_{*/}$ (note the change of variance!)
$$\colim_{i \in \N} \map((t_i, x_i), (g, x)) \to \map((t_{\infty}, x_{\infty}), (g, x))$$
is an equivalence. Using \eqref{Brown_pullback_square}, the last map of spaces is obtained from the equivalence
$$\colim_{i \in \N} \map(g, t_i) \xrightarrow{\simeq} \map(g, t_{\infty}),$$
viewed as a map over $F(g)$, by passing to the fibers over $x \in F(g)$.

This completes the proof that the opposite $\infty$-category of $F_{*/}$ satisfies the assumptions of Lemma \ref{criterion_C}. So, by Lemma \ref{criterion_C}, the $\infty$-category $F_{*/}$ admits an initial object, therefore $F$ is representable.  
\qed

\begin{remark} \label{clarify_terms_2}
We clarify that \cite[Theorem 5.2.7]{NRS} is not the special case of Theorem \ref{higher_Brown} for $n = \infty$ -- this issue about the terminology 
was also pointed out in Remark \ref{clarify_terms}. Instead, \cite[Theorem 5.2.7]{NRS} is a special case of Theorem \ref{higher_Brown} for $n=1$. 
\end{remark}


\begin{example} \label{pres_Brown}
By Theorem \ref{higher_Brown} and Example \ref{spaces_example}, the homotopy $n$-category of a finitely presentable $\infty$-category satisfies 
Brown representability for any $n \geq 2$. 
\end{example}

\begin{example} \label{stable_pres_Brown}
By Theorem \ref{higher_Brown} and Example \ref{spectra_example}, the homotopy $n$-category of a finitely presentable stable $\infty$-category satisfies Brown representability for any $n \geq 1$.  
\end{example}

\begin{corollary} \label{cor_Brown_rep_loc}
Let $\D$ be an $n$-category which is a localization of a compactly generated $n$-category, where $n \geq 1$ is an integer or $n=\infty$. Then $\D$ satisfies Brown representability.
\end{corollary}
\begin{proof}
This follows directly from Theorem \ref{higher_Brown} and Proposition \ref{local_preserves_Brown}.
\end{proof}

\subsection{Presentable $\infty$-categories} We will use the following general structure theorem for presentable $\infty$-categories from \cite{HTT, HA}:

\begin{theorem} \label{pres_structure_thm}
Every presentable (stable) $\infty$-category is equivalent to a localization of a finitely presentable (stable) $\infty$-category. 

As a consequence, the homotopy $n$-category of a presentable (stable) $\infty$-category is equivalent to a localization of the 
homotopy $n$-category of a finitely presentable (stable) $\infty$-category. 
\end{theorem}
\begin{proof}
The general case follows from \cite[Theorem 5.5.1.1]{HTT}. The stable case follows from 
\cite[Proposition 1.4.4.9]{HA}. 
\end{proof}

Combining Theorem \ref{pres_structure_thm} with our previous results, we obtain the following general class of examples of 
locally small weakly cocomplete $n$-categories which satisfy Brown representability (cf. Corollary \ref{cor_Brown_rep_loc}). 

\begin{corollary} \label{higher_Brown_2} Let $n \geq 1$ be an integer or $n=\infty$.
\begin{enumerate}
\item Suppose that $\C$ is a presentable stable $\infty$-category. Then $\h_n\C$ satisfies Brown representability.
\item Suppose that $\C$ is a presentable $\infty$-category. Then $\h_n\C$ satisfies Brown representability for any $n \geq 2$. 
\end{enumerate}
\end{corollary}
\begin{proof}
By Theorem \ref{pres_structure_thm} and Proposition \ref{local_preserves_Brown}, it suffices to prove (1)--(2) in the case where the $\infty$-category $\C$ is finitely presentable. This is the special case of Theorem \ref{higher_Brown} for the Examples \ref{pres_Brown} and \ref{stable_pres_Brown}. 
\end{proof}

We note that Corollary \ref{higher_Brown_2}(2) fails for $n=1$, e.g., it fails for the usual homotopy category of spaces (see \cite{He}). Combining Corollary \ref{higher_Brown_2} and Corollary \ref{Brown_implies_adjoints}, we obtain the following left adjoint functor theorem for homotopy $n$-categories of presentable $\infty$-categories.

\begin{corollary} \label{higher_Brown_3}
Let $\C$ be a presentable $\infty$-category and let $\D$ be a locally small $n$-category, where $n \geq 1$ is an integer or $n = \infty$.
\begin{enumerate}
\item Suppose that $\C$ is stable. Then a functor $F \colon \h_n \C \to \D$ admits a right adjoint 
if and only if $F$ preserves small coproducts and weak pushouts of order $(n-1)$.
\item Suppose that $n \geq 2$. Then a functor $F \colon \h_n \C \to \D$ admits a right adjoint if and only if $F$ preserves small coproducts and weak pushouts of order $(n-1)$.
\end{enumerate}
\end{corollary}
Note that the case $n=\infty$ recovers the left adjoint functor theorem for presentable $\infty$-categories \cite[Corollary 5.5.2.9(1)]{HTT}, \cite[Section 4]{NRS}.


\begin{thebibliography}{20}


\bibitem{AR}
J.~Ad\'{a}mek and J.~Rosick\'{y}, \emph{Locally presentable and accessible categories.} London Mathematical Society Lecture Note Series 189. Cambridge University Press, Cambridge, 1994. 

\bibitem{Ca}
K.~D.~Arlin, \emph{2-categorical Brown representability and the relation between derivators and infinity-categories.} PhD dissertation, University of California, Los Angeles, 2020. Available online at: \url{https://core.ac.uk/download/288432569.pdf}

\bibitem{AC}
K.~Arlin and J.~D.~Christensen, \emph{Detecting isomorphisms in the homotopy category.} Algebr. Geom. Topol. (to appear). arXiv: \url{https://arxiv.org/abs/1910.04141}. 

\bibitem{Br}
E.~H.~Brown, Jr., \emph{Cohomology theories.} Ann. of Math. (2) 75 (1962), 467--484.

\bibitem{CL}
A.~Campbell and E.~Lanari, \emph{On truncated quasi-categories.} Cah. Topol. G\'eom. Diff\'er. Cat\'eg. 61 (2020), no.~2, 154--207. 


\bibitem{Ci}
D.--C.~Cisinski, \emph{Higher categories and homotopical algebra.} Cambridge Studies in Advanced Mathematics Vol. 180. Cambridge University Press, Cambridge, 2019. 

\bibitem{DS1}
D.~Dugger and D.~I.~Spivak, \emph{Rigidification of quasi-categories.} Algebr. Geom. Topol. 11 (2011), no.~1, 225--261.

\bibitem{DS2}
D.~Dugger and D.~I .~Spivak, \emph{Mapping spaces in quasi-categories.} Algebr. Geom. Topol. 11 (2011), no.~1, 263--325.

\bibitem{FH}
P.~Freyd and A.~Heller, \emph{Splitting homotopy idempotents. II.}  J. Pure Appl. Algebra 89 (1993), no. 1--2, 93--106.

\bibitem{He}
A.~Heller, \emph{On the representability of homotopy functors.} J. London Math. Soc. (2) 23 (1981), no. 3, 551--562. 

\bibitem{Jo}
A.~Joyal, \emph{Quasi-categories and Kan complexes.} J. Pure Appl. Algebra 175 (2002), no. 1--3, 207--222.

\bibitem{HTT}
J.~Lurie, \emph{Higher topos theory.} Annals of Mathematics Studies Vol. 170. Princeton University Press, Princeton, NJ, 2009. Revised version available online 
at: \url{https://www.math.ias.edu/~lurie/papers/HTT.pdf}

\bibitem{HA}
J.~Lurie, \emph{Higher algebra.} Available online at: \url{https://www.math.ias.edu/~lurie/papers/HA.pdf}

\bibitem{ML}
S.~MacLane, \emph{Categories for the working mathematician}. Graduate Texts in 
Mathematics Vol. 5 (1971), Springer-Verlag, New York-Berlin.

\bibitem{Ne}
A.~Neeman, \emph{Triangulated categories.} Annals of Mathematical Studies Vol. 148. Princeton University Press, Princeton, NJ, 2001. 

\bibitem{NRS}
H.~K.~Nguyen, G.~Raptis, and C.~Schrade, \emph{Adjoint functor theorems for $\infty$-categories.}  
J. Lond. Math. Soc. (2) 101 (2020), no. 2, 659--681. 

\bibitem{Ra}
G.~Raptis, \emph{Higher homotopy categories, higher derivators, and $K$-theory.} Forum of Mathematics, Sigma, 10 (2022), e54, 1--36. 

\bibitem{RV}
E.~Riehl, D.~Verity, \emph{The 2-category theory of quasi-categories.} Adv. Math. 280 (2015), 549--642.



\end{thebibliography}
\end{document}